\documentclass[12pt, reqno]{amsproc}
\usepackage[utf8]{inputenc}
\usepackage[english]{babel}
\usepackage{amsmath}
\usepackage{amsfonts}
\usepackage{amssymb}
\usepackage{amsthm}
\usepackage{mathtools}
\usepackage{enumitem}

\usepackage[dvipsnames]{xcolor}

\usepackage{hyperref}
\hypersetup{
    hypertexnames=false,
    colorlinks,
    linkcolor={red!50!black},
    citecolor={blue!50!black},
    urlcolor={blue!80!black}
}

\usepackage{graphicx}
\usepackage[all,color]{xy}
\usepackage{soul}
\usepackage[margin=2.5cm]{geometry}

\newcommand\mycolor[1]{}

\newtheorem{lemma}[subsection]{Lemma}

\newtheorem{corollary}[subsection]{Corollary}

\newtheorem*{theorem*}{Theorem}
\newtheorem*{proposition*}{Proposition}

\newtheorem{subtheorem}[subsubsection]{Theorem}
\newtheorem{sublemma}[subsubsection]{Lemma}

\newtheorem{subcorollary}[subsubsection]{Corollary}

\newtheorem{subremark}[subsubsection]{Remark}

\newtheorem{subconjecture}[subsubsection]{Conjecture}

\makeatletter
\@addtoreset{subsection}{section}
\@addtoreset{equation}{section}
\@addtoreset{figure}{section}
\@addtoreset{table}{section}
\makeatother

\makeatletter
\newcommand\testshape{family=\f@family; series=\f@series; shape=\f@shape.}
\def\myemphInternal#1{\if n\f@shape%
\begingroup\itshape #1\endgroup\/%
\else\begingroup\sf\itshape\small #1\endgroup%
\fi}
\def\myemph{\futurelet\testchar\MaybeOptArgmyemph}
\def\MaybeOptArgmyemph{\ifx[\testchar \let\next\OptArgmyemph
                 \else \let\next\NoOptArgmyemph \fi \next}
\def\OptArgmyemph[#1]#2{\index{#1}\myemphInternal{#2}}
\def\NoOptArgmyemph#1{\myemphInternal{#1}}
\makeatother

\newcommand\term[2][\empty]{\myemph[#1]{#2}}

\newcommand\Eman{E}

\newcommand\Lman{L}
\newcommand\Mman{M}
\newcommand\Nman{N}

\newcommand\Qman{Q}

\newcommand\Uman{U}

\newcommand\Xman{X}
\newcommand\Yman{Y}

\newcommand\bC{\mathbb{C}}

\newcommand\bR{\mathbb{R}}
\newcommand\bZ{\mathbb{Z}}

\newcommand\id{\mathrm{id}}          
\newcommand\Int{\mathrm{Int\,}}        
\newcommand\rank{\mathsf{rank\,}}      

\newcommand\restr[2]{#1\vert_{#2}}
\newcommand\GL{\mathrm{GL}}

\newcommand\UnitGroup{\mathbf{1}}

\newcommand\Aut{\mathrm{Aut}}       
\newcommand\Diff{\mathcal{D}}       
\newcommand\DiffPl{\Diff^{+}}       
\newcommand\Orb{\mathcal{O}}        
\newcommand\Stab{\mathcal{S}}       
 
\newcommand\DiffId{\Diff_{\id}}     
\newcommand\StabId{\Stab_{\id}}     
\newcommand\RP[1]{\mathbb{RP}^{#1}}

\newcommand\Cinfty{\mathcal{C}^{\infty}}

\newcommand\Ci[2]{\mathcal{C}^{\infty}(#1,#2)}               
\newcommand\Cid[2]{\mathcal{C}_{\partial}^{\infty}(#1,#2)}   
\newcommand\Morse[2]{\mathcal{M}(#1,#2)}                     

\newcommand\Stabilizer[1]{\Stab(#1)}
\newcommand\StabilizerId[1]{\StabId(#1)}
\newcommand\StabilizerIsotId[1]{\Stab'(#1)}

\newcommand\StabilizerNbh[1]{\Stab_{{\nb}}(#1)}  

\newcommand\Orbit[1]{\Orb(#1)}
\newcommand\OrbitComp[2]{\Orb_{#1}(#2)}
\newcommand\OrbitPl[1]{\Orb^{+}(#1)}
\newcommand\OrbitPlComp[2]{\Orb^{+}_{#1}(#2)}

\newcommand\fixsymbol{}
\newcommand\invsymbol{\mathrm{inv}}
\newcommand\nbsymbol{\mathrm{nb}}
\newcommand\folsymbol{*}

\newcommand\DiffInv[3][\empty]{\Diff_{\invsymbol}(#2,#3\ifx\empty #1\relax\else,#1\fi)}
\newcommand\DiffPlInv[3][\empty]{\Diff^{+}_{\invsymbol}(#2,#3\ifx\empty #1\relax\else,#1\fi)}
\newcommand\DiffIdInv[3][\empty]{\Diff_{\id,\invsymbol}(#2,#3\ifx\empty #1\relax\else,#1\fi)}

\newcommand\DiffFix[3][\empty]{\Diff_{\fixsymbol}(#2,#3\ifx\empty #1\relax\else,#1\fi)}
\newcommand\DiffIdFix[3][\empty]{\Diff_{\id}(#2,#3\ifx\empty #1\relax\else,#1\fi)}

\newcommand\DiffPlFix[3][\empty]{\Diff^{+}_{\fixsymbol}(#2,#3\ifx\empty #1\relax\else,#1\fi)}

\newcommand\DiffNb[3][\empty]{\Diff_{\nbsymbol}(#2,#3\ifx\empty #1\relax\else,#1\fi)}
\newcommand\DiffIdNb[3][\empty]{\Diff_{\id,\nbsymbol}(#2,#3\ifx\empty #1\relax\else,#1\fi)}

\newcommand\DiffHFix[3][\empty]{\Diff^{0}_{\fixsymbol}(#2,#3\ifx\empty #1\relax\else,#1\fi)}

\newcommand\DiffHNb[3][\empty]{\Diff^{0}_{\nbsymbol}(#2,#3\ifx\empty #1\relax\else,#1\fi)}

\newcommand\DiffPlusFix[3][\empty]{\Diff^{+}_{\fixsymbol}(#2,#3\ifx\empty #1\relax\else,#1\fi)}

\newcommand\FDiff[2][\empty]{\Diff^{\folsymbol}(#2\ifx\empty #1\relax\else,#1\fi)}
\newcommand\FDiffFix[2][\empty]{\Diff^{\folsymbol}_{\fixsymbol}(#2\ifx\empty #1\relax\else,#1\fi)}
\newcommand\FDiffA[2][\empty]{\Diff^{=}(#2\ifx\empty #1\relax\else,#1\fi)}

\newcommand\VBAut[2][\empty]{\GL(#2\ifx\empty #1\relax\else,#1\fi)}
\newcommand\VBAutPl[2][\empty]{\GL^{+}(#2\ifx\empty #1\relax\else,#1\fi)}

\newcommand\DiffLP{\Diff}  
\newcommand\DiffLPInv[3][\empty]{\DiffLP_{inv}(#2,#3\ifx\empty#1\relax\else,#1\fi)}

\newcommand\DiffLPFix[3][\empty]{\DiffLP_{fix}(#2,#3\ifx\empty#1\relax\else,#1\fi)}

\newcommand\DiffLPNb[3][\empty]{\DiffLP_{nb}(#2,#3\ifx\empty#1\relax\else,#1\fi)}

\newcommand\func{f}
\newcommand\gfunc{g}
\newcommand\dif{h}
\newcommand\gdif{g}

\newcommand\px{x}

\newcommand\pu{u}

\newcommand\torus[1]{T^{#1}} 
\newcommand\Torus{T^2} 
\newcommand\Sphere{S^2} 
\newcommand\Mobius{\mathbb{M}}
\newcommand\Circle{S^1}
\newcommand\KleinBottle{\mathbb{K}}

\newcommand{\comment}[1]{}

\newcommand\gel{g}
\newcommand\hel{h}

\newcommand\lel{l}
\newcommand\kel{k}

\newcommand{\nb}{\mathrm{nb}}

\newcommand\Agrp{A}
\newcommand\Bgrp{B}

\newcommand\Ggrp{G}
\newcommand\Hgrp{H}

\newcommand\tang[2][\empty]{\mathsf{T}\ifx\empty\relax\else_{#1}\fi#2}  
\newcommand\tfib[1]{\tang[\!\mathsf{fib}]{#1}}                          

\newcommand\exd[9]{
 #1 \ar@{^(->}[r]\ar@{^(->}[d] & 
 #2 \ar@{^(->}[d]\ar@{->>}[r]  &
 #3 \ar@{^(->}[d] \\
 #4 \ar@{^(->}[r]\ar@{->>}[d] & 
 #5 \ar@{->>}[d]\ar@{->>}[r]  &
 #6 \ar@{->>}[d] \\
 #7 \ar@{^(->}[r] & 
 #8 \ar@{->>}[r]  &
 #9 
}

\newcommand\kMob[1]{\Mobius_{#1}}           
\newcommand\tMob[1]{C_{#1}}   

\newcommand\kSact{F}
\newcommand\tSact{\widehat{\kSact}}

\newcommand\genSSa{\sigma}
\newcommand\genD{\kappa}
\newcommand\genInv{\delta}

\newcommand\actmap{p}

\newcommand\xS{\mathcal{S}}
\newcommand\xD{\mathcal{D}}
\newcommand\xO{\mathcal{O}}
\newcommand\xSa{\mathcal{S}_{\alpha}}
\newcommand\xDa{\mathcal{D}_{\alpha}}
\newcommand\xOa{\mathcal{O}_{\alpha}}

\newcommand\xSM[1]{\mathcal{S}_{#1}}
\newcommand\xDM[1]{\mathcal{D}_{#1}}
\newcommand\xOM[1]{\mathcal{O}_{#1}}
\newcommand\xSanb{\mathcal{S}_{\alpha}^{\nb}}
\newcommand\xSMnb[1]{\mathcal{S}_{#1}^{\nb}}

\newcommand\bdSSa{\partial_{\xS,\xSa}}
\newcommand\bdDS{\partial_{\xD,\xS}}
\newcommand\bdDaSa{\partial_{\xDa,\xSa}}
\newcommand\incSa{i_{\xSa}}

\newcommand\incDaSa{i_{\xDa,\xSa}}

\newcommand\fogD{j_{\xD}}

\newcommand\rstfunc{\rho}
\newcommand\rstdif{\widehat{\rho}}

\newcommand\nbinc{\nu}

\newcommand\Bclass{\mathcal{B}}

\title[Smooth functions on Klein bottle]
{Smooth functions that split a Klein bottle into two M\"obius bands}

\author{Bohdan Mazhar}
\address{Department of Algebra and Topology, Institute of Mathematics of NAS of Ukraine, Teresh\-chenkivska str. 3, Kyiv, 01601, Ukraine}
\email{mazhar@imath.kiev.ua}

\author{Sergiy Maksymenko}
\address{Department of Algebra and Topology, Institute of Mathematics of NAS of Ukraine, Teresh\-chenkivska str. 3, Kyiv, 01601, Ukraine}
\curraddr{}

\email{maks@imath.kiev.ua}


\begin{document}

\begin{abstract}
Given a compact surface $M$, consider the right action $\mathcal{C}^{\infty}(M)\times\mathcal{D}(M)\to\mathcal{C}^{\infty}(M)$, $(f, h) \mapsto f\circ h$, of the group $\mathcal{D}(M)$ of $\mathcal{C}^{\infty}$ diffeomorphisms of $M$ on the space $\mathcal{C}^{\infty}(M)$ of $\mathcal{C}^{\infty}$ functions on $M$.
For $f\in\mathcal{C}^{\infty}(M)$ denote by $\mathcal{O}(f)$ its orbit, and by $\mathcal{O}_f(f)$ the path component of $\mathcal{O}(f)$ containing $f$.

The paper continues a series of computations by many authors of homotopy types of orbits $\mathcal{O}_f(f)$ of smooth functions on compact surfaces.
We provide here the computations of $\mathcal{O}_f(f)$ for a special class of functions $f\in\mathcal{C}^{\infty}(K)$ on the Klein bottle $K$ having the following properties: (i)~at each critical point $f$ is smoothly equivalent to some homogeneous polynomial (e.g.\ $f$ is Morse), and (ii)~there is a regular connected component $\alpha$ of a level set of $f$ such that $K\setminus\alpha$ is a disjoint union of two open M\"obius bands, with closures $M_1$ and $M_2$.
Let $f_i = f|_{M_i}$ be the restriction of $f$ to the M\"{o}bius band $M_i$, $i=1,2$, and $\mathcal{O}_{f_i}(f_i)$ be the path component of $f_i$ in its orbit with respect to the above action of $\mathcal{D}(M_i)$.
The possible homotopy types of $\mathcal{O}_{f_i}(f_i)$ are explicitly computed earlier. 
We prove that $\mathcal{O}_f(f)$ is homotopy equivalent to $\mathcal{O}_{f_1}(f_1) \times \mathcal{O}_{f_2}(f_2)$.
\end{abstract}
\maketitle


\newcommand\pr{p}

\newcommand\cclass[1]{\mathcal{C}^{\infty}(#1)}
\newcommand\cclassr[1]{\cclass{#1,\bR}}

\newcommand\fclass[1]{\mathcal{F}(#1)}
\newcommand\fclassr[1]{\fclass{#1,\bR}}

\newcommand\orb[1]{\mathcal{O}(#1)}
\newcommand\orbsub[2]{\mathcal{O}_{#1}(#2)}
\newcommand\stab[1]{\mathcal{S}(#1)}
\newcommand\stabsub[2]{\mathcal{S}_{#1}(#2)}
\newcommand\stabpr[1]{\mathcal{S}'(#1)}
\newcommand\diff[1]{\mathcal{D}(#1)}
\newcommand\diffid[1]{\mathcal{D}_{\id}(#1)}

\newcommand\orbf{\mathcal{O}(f)}
\newcommand\orbfsubf{\mathcal{O}_{f}(f)}
\newcommand\diffk{\diff{K}}
\newcommand\diffidk{\diffid{K}}

\newcommand\orbmb[1]{\orb{f_{#1}, \partial}}
\newcommand\orbmbnb[1]{\orb{f_{#1}, U_{#1}}}
\newcommand\orbmbsub[1]{\mathcal{O}_{f_{#1}}(f_{#1})}
\newcommand\orbmbnbsub[1]{\orbsub{f_{#1}}{\partial}{f_{#1}}}

\newcommand\orbone{\orbmb{1}}
\newcommand\orbtwo{\orbmb{2}}
\newcommand\orbonenb{\orbmbnb{1}}
\newcommand\orbtwonb{\orbmbnb{2}}

\newcommand\pii{\pi_1}
\newcommand\pio{\pi_0}
\newcommand\hombr[1]{[\{#1\}]}
\newcommand\ts{\tilde\sigma}

\newcommand\krgraphf{\Gamma_f}
\newcommand\graphproj{p_f}
\newcommand\interior{\operatorname{Int}}
\newcommand\f{$f$}
\newcommand\two{$2$}

\section{Introduction}
Let $\Mman$ be a smooth compact $2$-manifold (a surface).
Consider a natural right action of the group $\Diff(\Mman)$ of $\Cinfty$ diffeomorphisms of $\Mman$ on the space of $\Cinfty$ functions $\Ci{\Mman}{\bR}$ given by
\begin{align}\label{intro:equ:action}
	\Ci{\Mman}{\bR} \times \diff{M} \to \Ci{\Mman}{\bR}, &&
    (\func, \dif) \mapsto \func \circ \dif.
\end{align}
Let $\func\in\Ci{\Mman}{\bR}$.
Then the corresponding stabilizer and orbit of $\func$ will be denoted by
\begin{align*}
\Stabilizer{\func} &:= \{\dif\in\Diff(\Mman) \mid \func\circ\dif=\func\}, &
\Orbit{\func}      &:= \{\func\circ\dif \mid \dif\in\Diff(\Mman)\}.
\end{align*}
Let $\Delta_{\func}:=\{ \func^{-1}(c) \}_{c\in\bR}$ be the partition of $\Mman$ into level sets of $\func$.
The connected components of elements of $\Delta_{\func}$ will be called \term{countours}.
A contour is \term{critical} if it contains a critical point of $\func$, and \term{regular} otherwise.
Notice that
\[
    \Delta_{\func\circ\dif} = \{ \dif^{-1}(\func^{-1}(c)) \mid c\in\bR\} = \dif^{-1}(\Delta_{\func}),
    \qquad
    \dif\in\Diff(\Mman).
\]
Hence, one can study the action~\eqref{intro:equ:action} by looking on how the diffeomorphisms of $\Mman$ perturb the partition $\Delta_{\func}$.
In particular, every $\dif\in\Stabilizer{\func}$ leaves invariant each element of $\Delta_{\func}$, and therefore it yields a permutations of contours of each level set of $\func$ sending regular countours to regular, and critical to critical.
For that reason it will be convenient to say that diffeomorphisms from $\Stabilizer{\func}$ \term{preserve} $\func$.

We will endow all spaces of $\Cinfty$ maps with the corresponding $\Cinfty$ Whitney topologies.
Then $\Diff(\Mman)$ and $\Stabilizer{\func}$ are topological groups, and the above action is continuous.
In particular, path components of $\Diff(\Mman)$ (as well as of $\Stabilizer{\func}$ and $\Orbit{\func}$) are pairwise homeomorphic.
Denote by $\DiffId(\Mman)$ the identity path component of $\Diff(\Mman)$ (i.e.\ path component containing $\id_{\Mman}$),
by $\StabilizerId{\func}$ the identity path component of $\Stabilizer{\func}$, and by $\OrbitComp{\func}{\func}$ the path component of $\Orbit{\func}$ containing $\func$.

Let also $\Cid{\Mman}{\bR}$ be the subset of $\Ci{\Mman}{\bR}$ consisting of functions $\func\colon\Mman\to\bR$ being \emph{locally constant and regular on $\partial\Mman$}, i.e.
\begin{enumerate}
\item[(B)] $\func$ takes constant values on every component of the boundary $\partial\Mman$ and has no critical points on it,
\end{enumerate}
and $\fclassr{\Mman} \subset \Cid{\Mman}{\bR}$ be the subset of functions $\func$ satisfying additional \emph{homogeneity} condition:
\begin{enumerate}
\item[(H)] for every critical point $z\in\Mman$ of $\func$, the germ of $\func$ at $z$ is smoothly equivalent to the germ of some nonzero homogeneous polynomial $g_z\colon \bR^2\to\bR$ \term{without multiple factors}.
\end{enumerate}
The latter condition easily implies that every critical point of $\func\in\fclassr{\Mman}$ is isolated.
Hence, due to compactness of $\Mman$, every such $\func$ has only finitely many critical points.
Moreover, by Morse lemma, every non-degenerate critical point satisfies that condition (namely, it is equivalent to the polynomial $\pm x^2 \pm y^2$), and therefore $\fclassr{\Mman}$ contains the set $\Morse{M}{\bR}$ of all Morse functions, constituting an open and everywhere dense subset of $\Cid{\Mman}{\bR}$.
However, $\func\in\fclassr{\Mman}$ may have degenerate critical points.

The paper continues a series of partial computations by many authors of the homotopy types of stabilizers and orbits of functions from $\fclassr{\Mman}$, see
\cite{Maksymenko:AGAG:2006, Kudryavtseva:ConComp:VMU:2012, Kudryavtseva:SpecMF:VMU:2012,
     Kudryavtseva:MathNotes:2012, Kudryavtseva:MatSb:2013, MaksymenkoFeshchenko:UMJ:2014,
     MaksymenkoFeshchenko:MFAT:2015, MaksymenkoFeshchenko:MS:2015, Feshchenko:Zb:2015,
     Feshchenko:MFAT:2016, Feshchenko:PIGC:2019, Maksymenko:TA:2020,
     KravchenkoFeshchenko:MFAT:2020, KuznietsovaMaksymenko:PIGC:2020, Feshchenko:PIGC:2021,
     KuznietsovaSoroka:UMJ:2021, KuznietsovaMaksymenko:TMNA:2025}, see below.
For the moment, those calculations do not completely cover only certain subclasses of functions on $2$-sphere $\Sphere$, projective plane $\RP{2}$, and Klein bottle $\KleinBottle$, and the aim of the present paper is to \term{initiate the calculations for functions on $\KleinBottle$}.

\subsection{Main results}
For the convenience of the reader we will formulate now the main statements, and in next subsection will discuss the previous results.

The following lemma will be proved in Section~\ref{sect:proof:prop:top_decomposition}.
It claims that $\fclassr{\KleinBottle}$ splits into three subclasses depending on the qualitative structure of functions.

\begin{sublemma}\label{intro:lm:top_decomposition}
Let $\KleinBottle$ be a Klein bottle, $\func\in\fclassr{\KleinBottle}$, and $\krgraphf$ be the Kronrod--Reeb graph of $\func$.
Then exactly one of the following statements holds.
\begin{enumerate}[label={\rm(\alph*)}, nosep, itemsep=1ex, leftmargin=*]
\item\label{intro:klein_topmebius}
$\krgraphf$ is a \term{tree} and there exists a \term{regular} contour $\alpha$ of $\func$ such that the complement $\KleinBottle\setminus \alpha$ is a disjoint union of two open M\"{o}bius bands.

\item\label{intro:klein_topdisks}
$\krgraphf$ is a \term{tree} and there exists a unique \term{critical} contour $\beta$ of $\func$ such that the complement $\KleinBottle\setminus \beta$ is a disjoint union of finitely many open disks.

\item\label{into:graph:one_cycle}
$\krgraphf$ has a \term{unique cycle} $\Omega$.
\end{enumerate}
\end{sublemma}

Assume that $\func\in\fclassr{\KleinBottle}$ satisfies statement~\ref{intro:klein_topmebius} of Lemma~\ref{intro:lm:top_decomposition}, so there exists a regular contour $\alpha$ of $\func$, such that the complement $\KleinBottle\setminus \alpha$ is a disjoint union of two open M\"{o}bius bands, see Figure~\ref{fig:alpha}.
\begin{figure}[htbp]
\includegraphics[height=2cm]{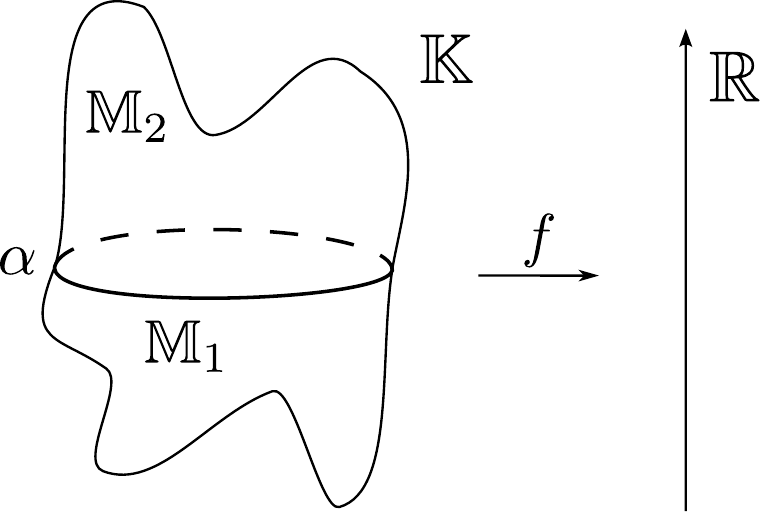}
\caption{Case~\ref{intro:klein_topmebius}: splitting of $\KleinBottle$ into M\"{o}bius bands}\label{fig:alpha}
\end{figure}
Denote by $\kMob{1}$ and $\kMob{2}$ their closures being closed M\"{o}bius bands and let $\func_i = \restr{\func}{\kMob{i}}$, $i=1,2$, be the corresponding restrictions of $\func$.
Since $\alpha$ is a regular contour of $\func$, it follows that $\func_i\in\fclass{\kMob{i}}$, and thus we can consider the orbit $\OrbitComp{\func_i}{\func_i}$ of $\func_i$ with respect to the action of $\Diff(\kMob{i})$.
Note that the possible homotopy types of orbits of functions on a M\"{o}bius band are explicitly described in~\cite{KuznietsovaMaksymenko:TMNA:2025}, see Section~\ref{subsect:algebraic-str-orbits} for more details.

The following Theorem~\ref{intro:th:acylcic_mobius_orbit} claim that $\OrbitComp{\func}{\func}$ is homotopy equivalent to $\OrbitComp{\func_1}{\func_1} \times \OrbitComp{\func_2}{\func_2}$.

Let $\DiffPlFix{\KleinBottle}{\alpha}$ be the group of diffeomorphisms of $\KleinBottle$ fixed on $\alpha$ and leaving invariant $\kMob{1}$ and $\kMob{2}$.
Denote by $\OrbitPl{\func,\alpha}$ the orbit of $\func$ with respect to the action of $\DiffPlFix{\KleinBottle}{\alpha}$, and by $\OrbitPlComp{\func}{\func,\alpha}$ the path component of $\func$ in $\OrbitPl{\func,\alpha}$ containing $\func$.
Then we have another well defined map
\begin{equation}\label{equ:rho_Ofa__Of1_x_Of2}
    \rho\colon \OrbitPl{\func,\alpha} \to \Orbit{\func_1} \times \Orbit{\func_2},
    \qquad
    \rho(\gfunc) = \bigl(\restr{\gfunc}{\kMob{1}},\restr{\gfunc}{\kMob{2}}\bigr).
\end{equation}
Indeed, let $\gfunc\in\Orbit{\func,\alpha}$.
This means that $\gfunc = \func\circ\dif$ for some $\dif\in\DiffFix{\KleinBottle}{\alpha}$.
Then $\dif_i:=\restr{\dif}{\Mobius_i} \in \DiffFix{\Mobius}{\alpha}$ is a diffeomorphism of $\Mobius_i$ fixed on its boundary $\alpha = \partial\Mobius_i$, whence
\[
    \restr{\gfunc}{\Mobius_i}=\restr{\func\circ\dif}{\Mobius_i} =
    \restr{\func}{\Mobius_i}\circ\restr{\dif}{\Mobius_i} = \func_i\circ\dif_i \in \Orbit{\func_i},
    \quad
    i=1,2.
\]

Since, $\rho(\func) = (\func_1,\func_2)$, we see that
\[
    \rho(\OrbitPlComp{\func}{\func,\alpha}) \subset
        \OrbitComp{\func_1}{\func_1} \times \OrbitComp{\func_2}{\func_2}
\]
Let also
\[ j\colon\OrbitPlComp{\func}{\func,\alpha}  \, \subset \, \Orbit{\func} \]
be the natural inclusion.

\begin{subtheorem}\label{intro:th:acylcic_mobius_orbit}
If $\func\in\fclassr{\KleinBottle}$ satisfies the statement~\ref{intro:klein_topmebius} of Lemma~\ref{intro:lm:top_decomposition} then both maps
\[
    \xymatrix{
        \ \OrbitComp{\func}{\func} \
            &
        \ \OrbitPlComp{\func}{\func,\alpha} \
            \ar[r]^-{\rho}
            \ar@{_(->}[l]_-{j}
            &
        \ \OrbitComp{\func_1}{\func_1} \times \OrbitComp{\func_2}{\func_2} \
    }
\]
are homotopy equivalences.
In particular, $\OrbitComp{\func}{\func}$ is homotopy equivalent to $\OrbitComp{\func_1}{\func_1} \times \OrbitComp{\func_2}{\func_2}$.
\end{subtheorem}

It is known that $\OrbitComp{\func}{\func}$ and $\OrbitComp{\func_i}{\func_i}$, $i=1,2$, are aspherical, see Corollary~\ref{cor:pik_Of} below, and therefore their homotopy types are determined by the corresponding fundamental groups.
The algebraic structure of groups $\pi_1\OrbitComp{\func_i}{\func_i}$ for functions on M\"{o}bius bands is described in~\cite{KuznietsovaMaksymenko:TMNA:2025}.
Then Theorem~\ref{intro:th:acylcic_mobius_orbit} allows to describe the algebraic structure of $\pi_1\OrbitComp{\func}{\func} \cong \pi_1\OrbitComp{\func_1}{\func_1} \times \pi_1\OrbitComp{\func_2}{\func_2}$, see Corollary~\ref{cor:K_pi1Of_case_a} below.

The proof of Theorem~\ref{intro:th:acylcic_mobius_orbit} will be given in Section~\ref{sect:proof:intro:th:acylcic_mobius_orbit}.
Namely, statement about $j$ is proved in Lemma~\ref{lm:inc_Oa_O}, while statement about $\rho$ is established in Lemma~\ref{lm:rest_Oa_O1_O2}.

The other two classes of functions $\func\in\fclass{\KleinBottle}$ from Lemma~\ref{intro:lm:top_decomposition} require different technique and will be investigated in another paper.

\section{Deformations of functions on surfaces by diffeomorphisms}
In what follows $\Mman$ will be a compact surface, $\Circle = \{|z|=1\} \subset \bC$ the unit circle in the complex plane, $\Mobius$ the M\"{o}bius band, $\KleinBottle$ the Klein bottle, $\torus{0}$ a one-point space, and $\torus{k} = (\Circle)^k$, $k\geq1$, is the \term{$k$-torus} (a product of $k$ circles).
Also the sign $\simeq$ means a \term{homotopy equivalence}.

We will discuss here known results about the homotopy types of stabilizers and orbits of functions from class $\fclass{\Mman}$.

Let $\Xman \subset \Mman$ be a closed subset.
Then we will denote by $\DiffFix{\Mman}{\Xman}$ the group of diffeomorphisms pointwise fixed on $\Xman$, by $\DiffNb{\Mman}{\Xman}$ be its subgroup consisting of diffeomorphisms $\dif$ each being fixed on some neighborhood of $\Xman$ (depending on $\dif)$, and by $\DiffIdFix{\Mman}{\Xman}$ its identity path component.
Let also $|\Xman|$ be the total number of points in $\Xman$, so it takes values in $\{0,1,\ldots,\infty\}$.

Furthermore, for $\func\in\Ci{\Mman}{\bR}$ we will use the following notations:
\begin{itemize}[leftmargin=*]
\item $\Sigma_{\func}$ is the set of all critical points of $\func$;
\item $\Stabilizer{\func,\Xman}$ and $\Orbit{\func,\Xman}$ are the corresponding stabilizer and orbit of $\func\in\Ci{\Mman}{\bR}$ under the induced action~\eqref{intro:equ:action} of $\DiffFix{\Mman}{\Xman}$;
\item $ \StabilizerNbh{\func,\Xman} =  \Stabilizer{\func} \cap \DiffNb{\Mman}{\Xman}$;
\item $\StabilizerId{\func,\Xman}$ and $\OrbitComp{\func}{\func,\Xman}$ are the corresponding path components of $\Stabilizer{\func,\Xman}$ and $\Orbit{\func,\Xman}$ containing $\id_{\Mman}$ and $\func$ respectively;
\item $\StabilizerIsotId{\func,\Xman} := \Stabilizer{\func} \cap \DiffIdFix{\Mman}{\Xman}$ is the group of diffeomorphisms of $\Mman$ preserving $\func$ and isotopic to identity via isotopy which \term{does not necessarily} preserve $\func$.
\end{itemize}
Note that, in contrast to $\StabilizerIsotId{\func,\Xman}$, the identity path component $\StabilizerId{\func,\Xman}$ consists of diffeomorphisms isotopic to identity via isotopy which \term{does} preserve $\func$.

\subsection{$\func$-regular neighborhoods}\label{subsect:f-reg-nbh}
Let $\func\in\fclassr{\Mman}$.
Recall that by a \term{contour} of $\func$ we mean a connected component of some level set $\func^{-1}(c)$, $c\in\bR$.

A subset $\Xman\subset\Mman$ will be called \term{$\func$-adapted} if it is a union of contours of $\func$.
One can observe that a compact $2$-dimensional submanifold $\Xman\subset\Mman$ is $\func$-adapted if and only if its boundary components are contours of $\func$.
A typical example of a $1$-dimensional $\func$-adapted submanifold is a finite union of contours.

Now if $\Xman\subset \Mman$ is an $\func$-adapted submanifold, then an \term{$\func$-regular neighborhood} of $\Xman$ is an any compact $\func$-adapted subsurface $\Uman\supset \Xman$ being a neighborhood of $\Xman$ and such that $\Uman\setminus \Xman$ contains no critical points of $\func$.

\begin{sublemma}[{\rm\cite[Corollary~7.2]{Maksymenko:TA:2020}}]
\label{lm:regular-invariant-neigh-homotopy-eq}
Let $\func\in\fclassr{\Mman}$, $\Xman$ be an $\func$-adapted submanifold, and $\Uman$ be any  $\func$-regular neighborhood of $\Xman$.
Then the inclusions
\[
    \Stabilizer{\func,\Uman}    \ \subset \
    \StabilizerNbh{\func,\Xman} \ \subset \
    \Stabilizer{\func,\Xman}
\]
are homotopy equivalences.
\end{sublemma}

\subsection{Graph of a function}\label{subsect:graph-action}
Consider the partition $\krgraphf$ of $\Mman$ into \term{contours}, and let $\graphproj\colon\Mman\to\krgraphf$ be the corresponding quotient map associating to each $x\in\Mman$ the contour $\graphproj(x)$ containing $x$.
Endow $\krgraphf$ with the quotient topology with respect to $\graphproj$, so a subset $A\subset\krgraphf$ is open iff its inverse image $\graphproj^{-1}(A)$ is open in $\Mman$.
Then $\krgraphf$ is called the \term{Kronrod--Reeb graph}, or simply the \term{graph} of $f$.

Notice that we have the following decomposition of $\func$:
\[
	\func \colon \Mman \xrightarrow{~\graphproj~} \krgraphf \xrightarrow{~\hat{\func}~} \bR,
\]
for a unique continuous function $\hat{\func}\colon\krgraphf\to\bR$.

Let $\Lman$ be the union of critical contours of $\func$.
Then it is easy to see that $\krgraphf$ has structure of $1$-dimensional CW complex whose \term{vertices} ($0$-cells) are critical contours of $\func$, while \term{edges} (open $1$-cells) correspond to connected components of $\Mman\setminus\Lman$.
Moreover, the function $\hat{\func}$ is strictly monotone on open edges of $\krgraphf$.

The following statement is well-known and follows from the easy observation that $\graphproj$ admits an ``homotopy right inverse''  map $s\colon\krgraphf\to\Mman$, i.e.\ $\graphproj\circ s$ is homotopic to $\id_{\krgraphf}$.
\begin{sublemma}\label{lm:fgraph_on_K}
For each $\func\in\fclass{\Mman}$, then induced homomorphism $\graphproj\colon H_1(\Mman,\bZ) \to H_1(\krgraphf,\bZ)$ of the first integral homologies is \term{surjective}.

In particular, if $\Mman=\KleinBottle$, then $\rank H_1(\krgraphf,\bZ) \leq 1$, so $\krgraphf$ is either a \term{tree} (has no cycles) or \term{contains a unique cycle}.
\end{sublemma}

\subsection{Homotopy types of diffeomorphism groups of surfaces}
Suppose that $\Xman \subset \Mman$ is a (possibly empty) union of some boundary components of $\Mman$ and finitely many points in $\Int{\Mman}$.
Then the homotopy types of $\DiffIdFix{\Mman}{\Xman}$ are well known and computed in~\cite{Smale:ProcAMS:1959, EarleEells:JGD:1969, EarleSchatz:DG:1970, Gramain:ASENS:1973}.
They are summarized in Table~\ref{tbl:hom_types:DidMX}, where $\Mobius$ denotes M\"{o}bius band.
\begin{table}[htbp!]
\begin{tabular}{|c|p{8cm}|c|}\hline
& \centering $\Mman$, $\Xman$ &  Homotopy type of $\DiffIdFix{\Mman}{\Xman}$ \\ \hline
1 & $\chi(\Mman) < |\Xman|$, this holds e.g.\ when              &  \\
  & $\bullet$~$\chi(\Mman)<0$, or                               & contractible \\
  & $\bullet$~$\Xman$ contains a boundary component of $\Mman$  &  \\ \hline
2 & $\bullet$~$\Mman=D^2, \Circle\times[0;1], \Mobius, \KleinBottle$, and $\Xman=\varnothing$     & \\
  & $\bullet$~$\Mman=D^2, \RP{2}, \Sphere$, and $|\Xman|=1$                                       & $S^1$ \\
  & $\bullet$~$\Mman=\Sphere$, and $|\Xman|=2$                                                    & \\ \hline
3 & $\Mman=\Torus$, and $\Xman=\varnothing$             & $T^2$ \\ \hline
4 & $\Mman=\Sphere, \RP{2}$, and $\Xman=\varnothing$    & $SO(3) = \RP{3}$ \\ \hline
\end{tabular}
\caption{Homotopy types of $\DiffIdFix{\Mman}{\Xman}$}
\label{tbl:hom_types:DidMX}
\end{table}

\subsection{Homotopy types of stabilizers and orbits of functions from $\fclass{\Mman}$}
Till the end of this section we will consider a triple
\begin{equation}\label{equ:MfX}
    (\Mman,\func,\Xman),
\end{equation}
in which $\Mman$ is a compact (not necessarily connected) surface, $\func\in\fclass{\Mman}$, $\Xman$ is a union of finitely many (possibly none) regular contours of $\func$ and some critical points of $\func$.

The starting result concerning the structure of orbits is the following
\begin{subtheorem}[{\rm\cite{Sergeraert:ASENS:1972, Maksymenko:AGAG:2006}}]\label{th:diff_act}
Let $(\Mman,\func,\Xman)$ be a triple~\eqref{equ:MfX}.
Then following statements hold.
\begin{enumerate}[leftmargin=*, label={\rm(\arabic*)}]
\item\label{enum:th:diff_act:Of_Freshet_manif}
The orbit $\Orbit{\func,\Xman}$ has structure of Fr\'{e}chet submanifold of finite codimension of Fr\'{e}chet space $\Cid{\Mman}{\bR}$.
In particular, $\Orbit{\func,\Xman}$ has the homotopy types of a CW-complex.

\item\label{enum:th:diff_act:DO_fibr}
The action map $\actmap\colon \DiffFix{\Mman}{\Xman} \to \Orbit{\func,\Xman}$, $\actmap(\dif) = \func\circ\dif$, is a locally trivial principal $\Stabilizer{\func,\Xman}$-fibration.

\item\label{enum:th:diff_act:DidOf_fibr}
$\actmap(\DiffIdFix{\Mman}{\Xman})=\OrbitComp{\func}{\func,\Xman}$ and the restriction $\actmap\colon\DiffIdFix{\Mman}{\Xman}\to\OrbitComp{\func}{\func,\Xman}$ is still a locally trivial principal $\StabilizerIsotId{\func,\Xman}$-fibration, where $\StabilizerIsotId{\func,\Xman} := \Stabilizer{\func} \cap \DiffIdFix{\Mman}{\Xman}$.

\item\label{enum:th:diff_act:OffX__Off}
$\OrbitComp{\func}{\func,\Xman} = \OrbitComp{\func}{\func,\Xman\cup\partial\Mman} = \OrbitComp{\func}{\func,\Xman\setminus\partial\Mman}$.
In particular, $\OrbitComp{\func}{\func} = \OrbitComp{\func}{\func,\partial\Mman}$ for $\Xman=\varnothing$.
\end{enumerate}
\end{subtheorem}
\begin{proof}[Remarks to the proof]
\ref{enum:th:diff_act:Of_Freshet_manif} and~\ref{enum:th:diff_act:DO_fibr} follow from~\cite[Th\'{e}or\`{e}me~8.1.1]{Sergeraert:ASENS:1972}, which is proved for the so-called \term{functions $\func$ of finite codimension} on arbitrary closed manifolds $\Mman$ and for the \term{left-right} action of $\Diff(\bR)\times\Diff(\Mman)$ on $\Ci{\Mman}{\bR}$ defined by $(\phi,\dif)\cdot\func = \phi\circ\func\circ\dif$.
The arguments of that theorem were adapted in~\cite[Theorem~2.1]{Maksymenko:AGAG:2006} to the \term{right} action~\eqref{intro:equ:action} of $\Diff(\Mman)$ and the case when $\Xman$ is a (possibly empty) collection of critical points.
Moreover, as mentioned in~\cite[Theorem~4.2]{KuznietsovaMaksymenko:TMNA:2025} that proof also holds for the case when $\Xman$ contains connected components of $\partial\Mman$.
In particular, when $\Mman$ is a compact surface, every $\func\in\fclass{\Mman}$ has ``finite codimension'', see e.g.~\cite[Lemma~4.7]{KuznietsovaMaksymenko:TMNA:2025}, so~\ref{enum:th:diff_act:Of_Freshet_manif} and~\ref{enum:th:diff_act:DO_fibr} hold for $\func$, and then $\Orbit{\func}$ will be a Fr\'{e}chet submanifold of $\Cid{\Mman}{\bR}$ of finite codimension.

Even more, one can easily extend the latter arguments to the situation when $\Xman$ is an arbitrary finite collections of contours and critical points.
In that case $\Orbit{\func}$ will be a Fr\'{e}chet submanifold of finite codimension of the linear space space $\mathcal{C}_{\Xman}(\Mman,\bR)$ consisting of $\Cinfty$ functions taking constant values on connected components of $\Xman$.

\ref{enum:th:diff_act:DidOf_fibr} is a direct consequence of path lifting property for $\actmap$.

\ref{enum:th:diff_act:OffX__Off} follows from~\ref{enum:th:diff_act:DidOf_fibr} and is based on axiom (B), see~\cite[Corollary~2.1]{Maksymenko:UMZ:ENG:2012}.
Let us briefly recall the proof.
Note that the inclusion $\OrbitComp{\func}{\func,\Xman\cup\partial\Mman} \subset \OrbitComp{\func}{\func,\Xman}$ is evident, since $\DiffIdFix{\Mman}{\Xman\cup\partial\Mman} \subset \DiffIdFix{\Mman}{\Xman}$.
Conversely, let $\gfunc\in\OrbitComp{\func}{\func,\Xman}$, i.e.\ $\gfunc=\func\circ\dif$ for some $\dif\in\DiffIdFix{\Mman}{\Xman}$.
Since $\dif$ is isotopic to $\id_{\Mman}$, it leaves invariant all components of $\partial\Mman$ and preserve their orientations.
In particular, $\restr{\dif}{\partial\Mman}$ is fixed on $\Xman$ and is isotopic rel.\ $\Xman$ to $\id_{\partial\Mman}$ via some isotopy $\phi_{t}$ such that $\phi_0 = \id_{\partial\Mman}$ and $\phi_1=\restr{\dif}{\partial\Mman}$.
That isotopy then extends to an ambient isotopy $\phi_t\colon\Mman\to\Mman$ fixed out some thin neighborhood of $\partial\Mman$ and leaving invariant the level sets of $\func$.
In other words, $\phi_t\in\Stabilizer{\func,\Xman}$, and $\phi_1^{-1}\circ\dif\in\DiffIdFix{\Mman}{\Xman\cup\partial\Mman}$.
But then $\gfunc = \func\circ\dif = \func\circ(\phi_1^{-1}\circ\dif) \in \OrbitComp{\func}{\func,\Xman\cup\partial\Mman}$.

Finally, since $(\Xman\setminus\partial\Mman)\cup \partial\Mman = \Xman\cup \partial\Mman$, we obtain that $\OrbitComp{\func}{\func,\Xman} = \OrbitComp{\func}{\func,\Xman\cup\partial\Mman} = \OrbitComp{\func}{\func,\Xman\setminus\partial\Mman}$.
\end{proof}

The next Theorem~\ref{th:hom_type:stabilizers} describes the homotopy types of path components of the stabilizers:
\begin{subtheorem}[\cite{Maksymenko:AGAG:2006, Maksymenko:OsakaJM:2011, Maksymenko:UMZ:ENG:2012}]
\label{th:hom_type:stabilizers}
Let $(\Mman,\func,\Xman)$ be a triple~\eqref{equ:MfX} such that $\Mman$ is connected.
If every critical point of $\func$ is a \term{non-degenerate local extreme} and $\Xman \subset \Sigma_{\func}$, then both groups  $\StabilizerId{\func,\Xman} \subset\DiffIdFix{\Mman}{\Xman}$ are homotopy equivalent to the circle.
In all other cases, $\StabilizerId{\func,\Xman}$ is contractible.
\end{subtheorem}
\begin{proof}[Remarks to the proof]
This theorem was initially proved for $\Xman=\varnothing$ in~\cite[Theorem~1.3]{Maksymenko:AGAG:2006} for Morse functions, in~\cite[Theorem~3.7]{Maksymenko:OsakaJM:2011} for all $\func\in\fclass{\Mman}$, in~\cite[Theorem~3]{Maksymenko:ProcIM:ENG:2010} for slightly larger classes of functions, and in~\cite[Theorem~2.1]{Maksymenko:UMZ:ENG:2012} it is extended to all variants of $\Xman$.
\end{proof}

Thus the case when $\StabilizerId{\func,\Xman}$ is homotopy equivalent to the circle covers exactly three types of Morse functions: on $\Sphere$ with only maximum and minimum, on $D^2$ with a single critical point, and on $\Circle\times[0;1]$ without critical points.

As a consequence of Theorems~\ref{th:diff_act} and~\ref{th:hom_type:stabilizers}, Table~\ref{tbl:hom_types:DidMX} and exact sequence of homotopy groups of fibration $\actmap\colon\DiffIdFix{\Mman}{\Xman}\to\OrbitComp{\func}{\func,\Xman}$, one easily deduces a lot of information about the homotopy types of the corresponding orbits.
For simplicity we will consider only the case when $\StabilizerId{\func,\Xman}$ is contractible.
\begin{subcorollary}[{\rm\cite[Theorem~1.5]{Maksymenko:AGAG:2006}}]
\label{cor:pik_Of}
Let $(\Mman,\func,\Xman)$ be a triple~\eqref{equ:MfX} such that $\Mman$ is connected and $\StabilizerId{\func,\Xman}$ is contractible.
Then for $k\geq2$
\begin{equation}\label{equ:pik_Of}
\pi_k\OrbitComp{\func}{\func,\Xman} =
\pi_k\DiffIdFix{\Mman}{\Xman} =
\begin{cases}
\pi_k SO(3), &  \text{if $(\Mman,\Xman) =(\Sphere,\varnothing) \ \text{or} \ (\RP{2},\varnothing)$}, \\
0, & \text{in all other cases},
\end{cases}
\end{equation}
and we also have the following short exact sequence:
\[
    1 \to
    \pi_1\DiffIdFix{\Mman}{\Xman} \xrightarrow{~\actmap~}
    \pi_1\OrbitComp{\func}{\func,\Xman} \xrightarrow{~\partial~}
    \pi_0\StabilizerIsotId{\func,\Xman} \to
    1
\]
In particular, if $\DiffIdFix{\Mman}{\Xman}$ is contractible as well, then we get an isomorphism
\begin{equation}\label{equ:partial_pi1Of__pi0Spr}
    \partial\colon\pi_1\OrbitComp{\func}{\func,\Xman} \cong \pi_0\StabilizerIsotId{\func,\Xman}.
\end{equation}
\end{subcorollary}

\begin{subremark}\rm
Regard $S^3$ as a unit sphere in $\bR^4$, and define the involution $\xi\colon S^3\to S^3$, $\xi(x)=-x$.
Then $SO(3)$ is homeomorphic to the quotient $S^3/\bZ_2$ being in turn a real projective space $\RP{3}$.
On the other hand, we also have a Hopf fibration $\psi\colon S^3\to S^2$ with fiber $S^1$.
Together this implies that $\pi_2 SO(3)=0$ and $\pi_k SO(3) = \pi_k S^3 = \pi_k S^2$ for $k\geq3$.
\end{subremark}

It will be convenient to formulate the following property in the form of conjecture:
\begin{subconjecture}\label{conj:g_act_on_tk}
As in Corollary~\ref{cor:pik_Of}, let $(\Mman,\func,\Xman)$ be a triple~\eqref{equ:MfX} such that $\Mman$ is connected and $\StabilizerId{\func,\Xman}$ is contractible.
Then there is $k\geq0$ and a free action of a certain finite group $G$ on $k$-torus $\torus{k}$ such that
\begin{equation}\label{equ:homtype_Of_TkG}
\OrbitComp{\func}{\func,\Xman} \simeq
\begin{cases}
(\torus{k}/G) \times SO(3), &  \text{if $(\Mman,\Xman) =(\Sphere,\varnothing) \ \text{or} \ (\RP{2},\varnothing)$}, \\
\torus{k}/G, & \text{in all other cases}.
\end{cases}
\end{equation}
\end{subconjecture}
This conjecture is verified for all $(\Mman,\func,\Xman)$ from \eqref{equ:MfX}, where either
\begin{enumerate}[label={\rm(\arabic*)}]
\item either $\Mman$ differs from $S^2$ and $\RP{2}$, or
\item $\Mman = S^2$ or $\RP{2}$ but $\func$ is Morse.
\end{enumerate}
Hence, it seems that it is also to be true for $\Mman=S^2$ and $\RP{2}$ and all $\func\in\fclass{\Mman}$.
\begin{proof}[Remarks to the proof]
1) In~\cite[Theorem~1.5]{Maksymenko:AGAG:2006} and \cite[Theorem~4]{Maksymenko:ProcIM:ENG:2010}, in addition to Corollary~\ref{cor:pik_Of}, it was also shown that for every $\func\in\fclass{\Mman}$ there is a short exact sequence
\begin{equation}\label{equ:exact_seq_pi1DidZn_Of_G}
    1 \to \pi_1\DiffId(\Mman) \oplus \bZ^{n} \to\pi_1\OrbitComp{\func}{\func} \to G \to 1
\end{equation}
for some $n\geq0$ and some finite group $G$.
That group $G$ is in fact the group of automorphisms of the ``enhanced'' Kronrod-Reeb graph of $\func$ induced by diffeomorphisms from $\StabilizerIsotId{\func}$, see \cite[Section~4]{Maksymenko:TA:2020} for details.
In the case when $G$ is trivial, \eqref{equ:pik_Of} and \eqref{equ:exact_seq_pi1DidZn_Of_G} imply~\eqref{equ:homtype_Of_TkG}, and thus $\OrbitComp{\func}{\func}$ is homotopy equivalent either to $\torus{n}$ or to $\torus{n}\times SO(3)$.
In particular that holds for \term{generic} Morse functions $\func$ (taking distinct values at distinct critical points) on all surfaces $\Mman$ and $\Xman=\varnothing$.
For such functions \cite[Theorem~1.5]{Maksymenko:AGAG:2006} gives a precise value of $n$.

2) Kudryavtseva studied the homotopy types of spaces of Morse functions on surfaces and, rediscovering arguments from~\cite{Maksymenko:AGAG:2006}, established~\eqref{equ:homtype_Of_TkG} for Morse functions $\func$ on orientable surfaces $\Mman$ with $\Xman\subset\Sigma_{\func}$, \cite{Kudryavtseva:SpecMF:VMU:2012},
\cite[Theorem~2.5]{Kudryavtseva:MathNotes:2012},
\cite[Theorem~2.6]{Kudryavtseva:MatSb:2013}, see also~\cite{Kudryavtseva:ENG:DAN2016, Kudryavtseva:sing:2021} for extensions to functions with other types of singularities.

3) Recall further that a group $B$ is called \term{crystallographic} if there is a short exact sequence $1\to \bZ^{k} \to B \to G \to 1$ for some $k\geq0$ and some finite group $G$.
Then, by the converse to Bieberbach theorem~\cite[Corollary~5.1]{Charlap:BiebGr:1986}, if $B$ is torsion free then there exists a free action of a group $G$ on a $k$-torus $\torus{k}$ such that $B \cong \pi_1(\torus{k}/G)$.
This is proved by Auslander and Kuranishi~\cite[Theorem~1]{AuslanderKuranishi:AM:1957} for the case when $\bZ^{k}$ corresponds to a maximal abelian subgroup of $B$ and extended to the general case by Vasquez~\cite[Theorem 3.1]{Vasquez:JDG:1970}.

Suppose $(\Mman,\Xman)$ differs from $(\Sphere,\varnothing)$ and $(\RP{2},\varnothing)$.
Then $\pi_1\DiffIdFix{\Mman}{\Xman}$ is a free abelian of rank $\leq 2$, whence~\eqref{equ:exact_seq_pi1DidZn_Of_G} implies that $\pi_1\OrbitComp{\func}{\func,\Xman}$ is crystallographic.
Moreover, in~\cite{Maksymenko:TrMath:2008} it is proved that for Morse functions $\func$ witn $n$ critical points, $\OrbitComp{\func}{\func,\Xman}=\OrbitComp{\func}{\func,\Xman\setminus\partial\Mman}$ has homotopy type of some $(2n-1)$-dimensional CW-complex.
Since $\OrbitComp{\func}{\func}$ is also aspherical, it follows from~\cite[Theorem~10.1]{Miller:AnnM:1984} that its fundamental group is torsion free.
Hence the existence of~\eqref{equ:homtype_Of_TkG} follows from the converse to Bieberbach theorem mentioned above.

Moreover, in the mentioned above series of papers, see Theorem~\ref{th:pi1Of_struct} below, it is computed an explicit structure of groups $\pi_1\OrbitComp{\func}{\func,\Xman}$ for all surfaces $\Mman$ expect for $\Sphere$, $\RP{2}$ and $\KleinBottle$ and $\func\in\fclass{\Mman}$.
In particular, see~\cite[Lemma~2.2]{Maksymenko:TA:2020}, it is reproved that $\pi_1\OrbitComp{\func}{\func,\Xman}$ is torsion free (actually for M\"{o}bius band this will be shown in Lemma~\ref{lm:torsion_free_groups} below, but the arguments are almost identical).

Also in~\cite[Corollary~6.2]{Maksymenko:EJM:2024} for all orientable $\Mman$ distinct from $\Sphere$ and all $\func\in\fclass{\Mman}$ it is construced an explicit actions of the corresponding group $G$ on a torus $\torus{k}$ yielding the homotopy equivalence~\eqref{equ:homtype_Of_TkG}.
\end{proof}

In next subsections we will describe the algebraic structure of $\pi_1\OrbitComp{\func}{\func,\Xman}$ for distinct cases of $\func$ and $\Xman$.

\subsection{Special kinds of wreath products}
\newcommand\aaa{a}
\newcommand\bbb{b}

Let $\Ggrp$ be a group with unit $e_{\Ggrp}$, and $\phi\colon \Ggrp \to \Ggrp$ be an automorphism.
Then we have the semidirect product $\Ggrp\rtimes_{\phi}\bZ$ being by definition the cartesian product $\Ggrp\times\bZ$ with respect to the following operation:
\[
    (\gel;\kel) (\hel;\lel) = (\gel \phi^{\kel}(\hel), \kel+\lel).
\]
The unit element in this group is $(e_{\Ggrp};0)$, and the inverse of $(\gel;\kel)$ is $(\phi^{-k}(\gel^{-1});-\kel)$.

More generally, if $\phi,\psi\colon \Ggrp \to \Ggrp$ are two commuting automorphisms, then we also have the semidirect product $\Ggrp\rtimes_{\phi,\psi}\bZ^2$ being by definition the cartesian product $\Ggrp\times\bZ^2$ with respect to the following operation:
\[
    (\gel;\kel,\lel) (\hel;\kel',\lel') =
    (\gel\, \phi^\kel(\psi^{\lel}(\hel)); \kel+\kel', \lel+\lel').
\]

Define now the following three types of groups.

1) Let $\aaa\geq1$, and $\phi\colon \Ggrp^{\aaa}\to \Ggrp^{\aaa}$, $\phi(\gel_0,\gel_1,\ldots,\gel_{\aaa-1}) = (\gel_1,\gel_2,\ldots,\gel_{\aaa-1},\gel_0)$, be the cyclic shift to the left.
Then the corresponding semidirect product $\Ggrp\rtimes_{\phi}\bZ$ will be denoted by $\Ggrp\wr_{\aaa}\bZ$.

2) Let $\aaa,\bbb\geq1$.
Consider the product $\Ggrp^{\aaa\bbb}$ of $\aaa\bbb$ copies of $\Ggrp$.
Its elements can be regarded as $\aaa\times \bbb$ matrices $\{\gel_{i,j}\}_{i\in\bZ_{\aaa}, j\in\bZ_{\bbb}}$ with entries in $\Ggrp$.

Then we have the following two automorphisms of $\Ggrp^{\aaa\bbb}$:
\begin{align*}
    &\phi,\psi\colon \Ggrp^{\aaa\bbb}\to \Ggrp^{\aaa\bbb}, &
    &\phi(\{ \gel_{i,j} \}) = \{ \gel_{i+1, j}\},&
    &\psi(\{ \gel_{i,j} \}) = \{ \gel_{i, j+1}\},
\end{align*}
making cyclic shifts of rows and columns respectively, where the first index is taken modulo $\aaa$ and the second one modulo $\bbb$.
Evidently, $\phi$ and $\psi$ commute and therefore we have a semidirect product $\Ggrp^{\aaa,\bbb}\rtimes_{\phi,\psi}\bZ^2$ which we will denote by $\Ggrp\wr_{\aaa,\bbb}\bZ^2$.

3) Finally, let $\Hgrp$ be another group with unit $e_{\Hgrp}$, $\gamma\colon \Hgrp \to \Hgrp$ be an automorphism of order $2$, and $\aaa\geq1$.
Consider the following automorphism $\phi\colon \Ggrp^{2\aaa}\times \Hgrp^{\aaa} \to \Ggrp^{2\aaa}\times \Hgrp^{\aaa}$,
\[
    \phi(\gel_0,\ldots,\gel_{2\aaa-1}; \ \hel_0,\ldots,\hel_{\aaa-1}) =
        \bigl(\gel_1,\gel_2,\ldots,\gel_{0}; \ \hel_1,\hel_2,\ldots,\hel_{\aaa-1},\gamma(\hel_0)\bigr).
\]
Evidently, it has order $2\aaa$.
The corresponding semidirect product $(\Ggrp^{2\aaa}\times \Hgrp^{\aaa}) \rtimes_{\phi}\bZ$ will be denoted by $(\Ggrp,\Hgrp)\wr_{\aaa,\gamma}\bZ$.

\begin{sublemma}\label{lm:torsion_free_groups}
If $\Ggrp$ is torsion free, then for every $\aaa,\bbb\geq1$ the groups $\Ggrp \wr_{\aaa} \bZ$ and $\Ggrp \wr_{\aaa,\bbb} \bZ^2$ are torsion free as well.

If $\Ggrp,\Hgrp$ are torsion free, then for every $\aaa$ and an automorphism $\gamma\colon \Hgrp \to \Hgrp$ of order $2$, the group $(\Ggrp,\Hgrp)\wr_{\aaa,\gamma}\bZ$ is torsion free as well.
\end{sublemma}
\begin{proof}
The proof for $\Ggrp \wr_{\aaa} \bZ$ is given in~\cite[Lemma~2.2]{Maksymenko:TA:2020}.
The proofs for $\Ggrp \wr_{\aaa,\bbb} \bZ^2$ and $(\Ggrp,\Hgrp)\wr_{\aaa,\gamma}\bZ$ is similar.
For the convenience of the reader let us present a proof for $(\Ggrp,\Hgrp)\wr_{\aaa,\gamma}\bZ$.

Suppose $q = (\gel_0,\ldots,\gel_{2m-1};\hel_0,\ldots,\hel_{\aaa-1};\kel)\in (\Ggrp,\Hgrp)\wr_{\aaa,\gamma}\bZ$ has finite order, say $\lel$.
Then
\[
   (e_{\Ggrp}, \ldots, e_{\Ggrp}; e_{\Hgrp},\ldots,e_{\Hgrp}; 0) =
   e =
   q^{\lel} = (\gel'_0,\ldots,\gel'_{2m-1};\hel'_0,\ldots,\hel'_{\aaa-1};\kel\lel)
\]
for some $\gel'_i\in\Ggrp$ and $\hel'_j\in\Hgrp$.
Hence $\kel\lel=0$, and therefore $\kel=0$ since $\lel\not=0$.
But then $\gel'_{i} = \gel_i^{\kel} = e_{\Ggrp}$ and $\hel'_{j} = \hel_j^{\kel} = e_{\Hgrp}$.
Since $\Ggrp$ and $\Hgrp$ are torsion free, $\gel_i=e_{\Ggrp}$ and $\hel_j=e_{\Hgrp}$, i.e. $q = e$.
\end{proof}

\subsection{Algebraic structure of fundamental groups of orbits}\label{subsect:algebraic-str-orbits}

\newcommand\PO[2][\empty]{\pi^{#1}_{#2}}

Let $\Yman\subset\Mman$ be an $\func$-adapted (not necessarily connected) subsurface, that is every boundary component of $\Yman$ is a regular contour of $\func$.
This implies that the restriction $\restr{\func}{\Yman}$ belongs to the class $\fclass{\Yman}$, and Theorems~\ref{th:diff_act} and \ref{th:hom_type:stabilizers} and Corollary~\ref{cor:pik_Of} are applicable for $\restr{\func}{\Yman}$.
It will be convenient to denote by
\[
    \PO[\func]{\Yman} := \pi_1\OrbitComp{\restr{\func}{\Yman}}{\restr{\func}{\Yman}} =
                  \pi_1\OrbitComp{\restr{\func}{\Yman}}{\restr{\func}{\Yman}, \partial\Yman}
\]
the fundamental group of the orbit of the restriction $\restr{\func}{\Yman}$ of $\func$ to $\Yman$, where the base point of $\PO[\func]{\Yman}$ is $\restr{\func}{\Yman}$.
In particular, $\PO[\func]{\Mman} := \pi_1\bigl(\OrbitComp{\func}{\func}, \func\bigr)$.
It is easy to see that if $\Yman = \mathop{\sqcup}\limits_{i=1}^{k}\Yman_i$ is a disjoint union of $\func$-adapted subsurfaces, then $\PO[\func]{\Yman} \cong \prod_{i=1}^{k} \PO[\func]{\Yman_i}$.

For simplicity, if $\func$ is fixed we will omit it from notation and denote $\PO[\func]{\Yman}$ simply by $\PO{\Yman}$.

\newcommand\sC[1][\empty]{\mathsf{C}_{#1}}
\newcommand\sD[1][\empty]{\mathsf{D}_{#1}}
\newcommand\sE[1][\empty]{\mathsf{E}_{#1}}
\newcommand\sM[1][\empty]{\mathsf{M}_{#1}}


\newcommand\lC[1][]{\mathsf{C}\ifx#1\relax\relax\else(#1)\fi}
\newcommand\lD[1][]{\mathsf{D}\ifx#1\relax\relax\else(#1)\fi}
\newcommand\lE[1][]{\mathsf{E}\ifx#1\relax\relax\else(#1)\fi}
\newcommand\lM[1][]{\mathsf{M}\ifx#1\relax\relax\else(#1)\fi}
\newcommand\CDEM[4]{%
    \expandafter\def\csname qqq \endcsname{\empty}
    \ifx#1\empty
        \relax
    \else
        \lC[#1]
        \expandafter\def\csname qqq \endcsname{.}
    \fi%
    \ifx#2\empty
        \relax
    \else
        \csname qqq \endcsname
        \lD[#2]
        \expandafter\def\csname qqq \endcsname{.}
    \fi%
    \ifx#3\empty
        \relax
    \else
        \csname qqq \endcsname
        \lE[#3]
        \expandafter\def\csname qqq \endcsname{.}
    \fi%
    \ifx#4\empty
        \relax
    \else
        \csname qqq \endcsname
        \lM[#4]
    \fi%
}

Say that $\Yman$ is \term{$\func$-proper}, if $\Mman\setminus\Yman$ contains some critical points of $\func$, i.e.\ \term{not all} critical points of $\func$ belong to $\Yman$.

Also, we will call $\Yman$ to be a \term{$\CDEM{c}{d}{e}{m}$-subsurface} for some $c,d,e,m\geq0$ whenever it is a disjoint union of
\begin{itemize}[nosep]
\item $c$ cylinders $\sC[1],\ldots,\sC[c]$,
\item $d$ disks $\sD[1],\ldots,\sD[d]$,
\item another collection of $e$ disks $\sE[1],\ldots,\sE[e]$,
\item and $m$ M\"{o}bius bands $\sM[1],\ldots,\sM[m]$.
\end{itemize}
In this case we also denote:
\begin{align*}
 \sC &= \mathop{\sqcup}\limits_{i=1}^{c} \sC[i],&
 \sD &= \mathop{\sqcup}\limits_{i=1}^{d} \sD[i],&
 \sE &= \mathop{\sqcup}\limits_{i=1}^{e} \sE[i],&
 \sM &= \mathop{\sqcup}\limits_{i=1}^{m} \sM[i].
\end{align*}
Then
\[
    \PO{\Yman} \cong \PO{\sC}\times\PO{\sD}\times\PO{\sE}\times\PO{\sM}
               \cong \prod\limits_{i=1}^{c} \PO{\sC[i]}\times
                     \prod\limits_{i=1}^{d} \PO{\sD[i]}\times
                     \prod\limits_{i=1}^{e} \PO{\sE[i]}\times
                     \prod\limits_{i=1}^{m} \PO{\sM[i]}.
\]
If some of numbers $c,d,e,m$ is zero, then we will omit the corresponding letter from notation.
For instance, a $\CDEM{\empty}{4}{\empty}{6}$-surface is a disjoint union of $4$ disks and $6$ M\"{o}bius bands, while $\CDEM{1}{\empty}{\empty}{\empty}$-subsurface is just a single cylinder.

Also, if some of the numbers $c,d,e,m$ are non-essential, then we omit them as well.
For instance, $\CDEM{}{}{\empty}{\empty}$-subsurface is a disjoint union of several cylinders and disks.

The following Theorem~\ref{th:pi1Of_struct} summarizes the information found in previous papers about the algebraic structure of $\PO{\Mman}$ for $\Mman$ distinct from $2$-sphere, projective plane and Klein bottle.
The main observation is that: (a)~$\PO{\Mman}$ can be explicitly expressed via $\PO{\Yman_0}$ for a certain $\CDEM{}{}{\empty}{\empty}$-subsurface of $\Mman$;
(b)~if $\Yman_0$ contains the minimal possible number of critical points, then one can directly compute $\PO{\Yman_0}$; (c)~otherwise, $\PO{\Yman_0}$ is expressed via $\PO{\Yman_1}$ for some \term{$\func$-proper} $\CDEM{}{}{\empty}{\empty}$-subsurface $\Yman_1\subset \Yman_0$, i.e.\ containing less number of critical points than $\Yman_0$, and the computation of $\PO{\Yman_1}$ reduces to the case (b) by induction.

\begin{subtheorem}\label{th:pi1Of_struct}
Let $\Mman$ be a compact surface, $\func\in\fclass{\Mman}$.
\begin{enumerate}[leftmargin=*, label={\rm(\arabic*)}]
\item\label{enum:p1Of:chi_neg}
It $\chi(\Mman)<0$, then $\PO{\Mman} \cong \PO{\Yman}$ for some $\CDEM{}{}{\empty}{}$-subsurface $\Yman$.

\item\label{enum:p1Of:torus}
Suppose $\Mman=\Torus$ is a $2$-torus.
\begin{enumerate}[leftmargin=*, label={\rm(\alph*)}]
\item\label{enum:p1Of:torus:tree}
If $\krgraphf$ is a tree, then $\PO{\Mman} \cong \PO{\Yman} \wr_{\aaa,\bbb}\bZ^2$ for some \term{$\func$-proper} $\CDEM{\empty}{}{\empty}{\empty}$-subsurface $\Yman$ and $\aaa,\bbb\geq1$.

\item\label{enum:p1Of:torus:cycle}
Otherwise, $\krgraphf$ has a unique cycle, and $\PO{\Mman} \cong \PO{\Yman} \wr_{\aaa}\bZ$ for some $\CDEM{1}{\empty}{\empty}{\empty}$-subsurface $\Yman$ (being thus a single cylinder), and $\aaa\geq1$.
\end{enumerate}

\item\label{enum:p1Of:mobius}
If $\Mman = \Mobius$ is a M\"{o}bius band, then there exist
\begin{enumerate}[leftmargin=*, label={\rm(\alph*)}]
\item\label{enum:p1Of:mobius:de}
either an \term{$\func$-proper} $\CDEM{1}{}{}{\empty}$-subsurface, $\aaa\geq1$, and an involution $\gamma\colon\PO{\sE}\to\PO{\sE}$ such that
\[\PO{\Mman} \cong \PO{\sC} \, \times \, \bigl((\PO{\sD},\PO{\sE}) \wr_{\aaa,\gamma} \bZ\bigr);\]

\item\label{enum:p1Of:mobius:d}
or an \term{$\func$-proper} $\CDEM{1}{}{\empty}{\empty}$-subsurface and \term{odd} $\bbb\geq1$ such that
$\PO{\Mman} \cong \PO{\sC} \times \bigl(\PO{\sD}\wr_{\bbb} \bZ\bigr)$.
\end{enumerate}

\item\label{enum:p1Of:disk}
Suppose $\Mman=D^2$ is a $2$-disk.
If $\func$ has only one critical point $\px$ (being therefore a local extreme), then $\PO{\Mman} = 0$ when $\px$ is non-degenerate, and $\PO{\Mman}\cong\bZ$ when $\px$ is degenerate.

Otherwise, there exists $\aaa\geq1$ and
\begin{enumerate}[leftmargin=*, label={\rm(\alph*)}]
\item\label{enum:p1Of:disk:de}
either an \term{$\func$-proper} $\CDEM{\empty}{}{1}{\empty}$-subsurface such that $\PO{\Mman} \cong \PO{\sE} \times \bigl(\PO{\sD}\wr_{\aaa} \bZ\bigr)$, or

\item\label{enum:p1Of:disk:d}
or an \term{$\func$-proper} $\CDEM{\empty}{}{\empty}{\empty}$-subsurface such that $\PO{\Mman} \cong \PO{\sD}\wr_{\aaa} \bZ$.
\end{enumerate}

\item\label{enum:p1Of:cylinder}
Finally, suppose that $\Mman=\Circle\times[0;1]$ is a cylinder.
If $\func$ has no critical points, then $\PO{\Mman} = 0$.
Otherwise, $\PO{\Mman} \cong \PO{\sC} \times \bigl(\PO{\sD}\wr_{\aaa} \bZ\bigr)$ for some \term{$\func$-proper} $\CDEM{1}{}{\empty}{\empty}$-subsurface and $\aaa\geq1$.
\end{enumerate}
\end{subtheorem}
\begin{proof}[Remarks to the proof]
Statement~\ref{enum:p1Of:chi_neg} is proved in~\cite[Theorem~1.7]{Maksymenko:MFAT:2010}.
Statement~\ref{enum:p1Of:torus}\ref{enum:p1Of:torus:tree} is established in~\cite[Theorem~1]{MaksymenkoFeshchenko:UMJ:2014} for $\aaa=\bbb=1$ and in~\cite[Theorem~2.5]{Feshchenko:ProcIM:2015} for the general case, while~\ref{enum:p1Of:torus}\ref{enum:p1Of:torus:cycle} is shown in~\cite[Theorem~2]{MaksymenkoFeshchenko:MS:2015} for $\aaa=1$ and in~\cite[Theorem~1.6]{MaksymenkoFeshchenko:MFAT:2015} for $\aaa\geq1$.
Statement~\ref{enum:p1Of:mobius} proved in~\cite[Theorem~2.6]{KuznietsovaMaksymenko:TMNA:2025}, while~\ref{enum:p1Of:disk} and~\ref{enum:p1Of:cylinder} are shown in~\cite[Theorems~5.5, 5.6, 5.8]{Maksymenko:TA:2020}.
Note that the constructions of the subsurfaces $\Yman$, $\sC$, $\sD$, $\sE$ and $\sM$ in those papers are explicit.
\end{proof}

Let $\Bclass$ be the set of isomorphism classes of groups being minimal with respect to the following properties:
\begin{enumerate}[label={\rm(\alph*)}]
\item the unit group $\UnitGroup=\{1\}$ belongs to $\Bclass$;
\item if $G,H\in\Bclass$ and $n\geq1$, then $G\times H$ and $G \wr_{n}\bZ\in\Bclass$.
\end{enumerate}
Minimality of the class $\Bclass$ is equivalent to the assumption it consists of groups obtained from the unit group $\UnitGroup$ by finitely many operations of direct product and wreath products $\wr_{n}\bZ$.
For instance, $\Bclass$ contains the following groups: $\bZ = 1\wr_{1}\bZ$, $\bZ^n$ for all $n\geq1$, $(\bZ^{7}\wr_{3}\bZ)\wr_{11} \bZ$ etc.

Now Theorem~\ref{th:pi1Of_struct} implies that
\begin{enumerate}[leftmargin=*, label={\rm(\roman*)}]
\item\label{enum:classB:disk_et_all}
if either $\Mman = D^2$ or $\Mman=\Circle\times[0;1]$,
or $\Mman$ is orientable with $\chi(\Mman)<0$, then $\PO{\Mman} \in \Bclass$;
\item\label{enum:classB:torus}
if $\Mman = \Torus$, then there exists $\Hgrp\in\Bclass$ such that either $\PO{\Mman} \cong \Hgrp\wr_{\aaa}\bZ$ or $\Hgrp\wr_{\aaa,\bbb}\bZ^2$ for some $\aaa,\bbb\geq1$ and $\Hgrp\in\Bclass$, (in the first case $\PO{\Mman}\in\Bclass$ as well);
\item\label{enum:classB:mobius}
if $\Mman = \Mobius$ is either a M\"{o}bius band or $\Mman$ is non-orientable with $\chi(\Mman)<0$, then either $\PO{\Mman} \in \Bclass$ or $\PO{\Mman} \cong \Agrp \times \bigl(\prod_{i=1}^{g}(\Ggrp_i,\Hgrp_i)\wr_{\aaa_i,\gamma_i} \bZ\bigr)$ for some $\Agrp,\Ggrp_i,\Hgrp_i\in\Bclass$, $\aaa_i\geq1$, and involutions $\gamma_i\colon\Hgrp_i\to\Hgrp_i$, where $g$ does not exceed the \term{non-orientable genus} of $\Mman$ (i.e.\ the maximal number of mutually disjoint M\"{o}bius bands inside $\Mman$).
\end{enumerate}
Conversely, see~\cite[Theorem~1.2]{KuznietsovaSoroka:UMJ:2021} and \cite[Theorem~8.3.3]{Maksymenko:ProcIMNANU:2020}, for every surface $\Mman$ from the cases~\ref{enum:classB:disk_et_all} and~\ref{enum:classB:torus} and for every group $\Ggrp\in\Bclass$ and $\aaa,\bbb\geq1$ there exist $\func\in\fclass{\Mman}$ (which can even assumed to be Morse) such that $\PO[\func]{\Mman}\cong\Ggrp$.
Moreover, for any $\Hgrp\in\Bclass$ and $\aaa,\bbb\geq1$, there exist $\func\in\fclass{\Torus}$ (which can also assumed to be Morse) such that $\PO[\func]{\Mman}\cong\Hgrp\wr_{\aaa,\bbb}\bZ^2$.

As a consequence of Theorem~\ref{th:pi1Of_struct} and our main result Theorem~\ref{intro:th:acylcic_mobius_orbit} we get
\begin{corollary}\label{cor:K_pi1Of_case_a}
Let $\func\in\fclass{\KleinBottle}$ be a function satisfying the statement~\ref{intro:klein_topmebius} of Lemma~\ref{intro:lm:top_decomposition}.
Then $\OrbitComp{\func}{\func}$ is a $K(N,1)$-space, where $N := \pi_1\OrbitComp{\func}{\func} \cong \Agrp\times \Bgrp_1 \times \Bgrp_2$, $\Agrp\in\Bclass$, and each $\Bgrp_i$ is of one of the forms:
\begin{itemize}
   \item $\Ggrp_i\wr_{\bbb_i}\bZ$ for some odd $\bbb_i$ and $\Ggrp_i\in\Bclass$;
   \item $(\Ggrp_i,\Hgrp_i)\wr_{\aaa_i,\gamma_i}\bZ$ for some $\Ggrp_i,\Hgrp_i\in\Bclass$, $\aaa_i\geq1$, and an involution $\gamma_i\colon\Hgrp_i\to\Hgrp_i$.
   \qedhere
\end{itemize}
\end{corollary}

\section{Preliminaries}\label{sect:preliminaries}

\subsection{Diffeomorphism groups}
All manifolds are assumed to be of class $\Cinfty$.
Let $\Mman$ be a manifold.
Denote by $\Diff(\Mman)$ the group of all $\Cinfty$ diffeomorphisms of $\Mman$, and for a closed subset $\Xman\subset\Mman$ define also the following subgroups of $\Diff(\Mman)$.
Let
\begin{itemize}[leftmargin=*]
\item $\DiffInv{\Mman}{\Xman}$ be the group of diffeomorphisms $\dif$ of $\Mman$ leaving $\Xman$ invariant, i.e.\ $\dif(\Xman)=\Xman$;
\item $\DiffFix{\Mman}{\Xman}$ be its subgroup fixing $\Xman$ pointwise;
\item $\DiffNb{\Mman}{\Xman}$ be another subgroup consisting of diffeomorphisms $\dif$ each being fixed on some neighborhood of $\Xman$ (depending on $\dif)$.
\end{itemize}
Thus, we have the following inclusions:
\begin{align*}
    \DiffNb{\Mman}{\Xman}         \ \subset \
    \DiffFix{\Mman}{\Xman}        \ \subset \
    \DiffInv{\Mman}{\Xman}        \ \subset \
    \Diff(\Mman).
\end{align*}

A $\Cinfty$ map $\gamma\colon\Nman\to\Mman$ between two manifolds is an \term{embedding} if the induced tangent map $T\gamma\colon T\Nman\to T\Mman$ is injective.
Then by a \term{submanifold} of $\Mman$ we will mean the image of some embedding.

Denote by $\DiffPlFix{\Mman}{\Xman}$ the subgroup of $\DiffFix{\Mman}{\Xman}$ consisting of diffeomorphisms $\dif$ which also ``preserve the sides near $\Xman$'', i.e.\ if $\Yman$ is a connected component of $\Xman$, $\Yman\subset\Uman\subset\Eman_{\Yman}$ where $\Uman$ is any open neighborhood of $\Yman$ and $\Eman_\Yman$ is some tubular neighborhood of $\Yman$, $\Eman_{\Yman}^{0}$ and $\Eman_{\Yman}^{1}$ are connected components of $\Eman_{\Yman}\setminus\Yman$, with $\dif(\Uman)\subset\Eman_{\Yman}$, then $\dif(\Uman\cap\Eman_{\Yman}^{i}) \subset \Eman_{\Yman}^{i}$ for $i=0,1$.

We will need the following theorem, which seems to be known however the authors did not find a formal proof in the literature:
\begin{subtheorem}\label{th:DnbMX__DplFMX_he}
Let $\Xman$ be a compact submanifold of $\Mman$ such that every its connected component~$\Yman$ is
\begin{enumerate}[label={\rm(\arabic*)}]
\item\label{enum:X:bd}
either a boundary component of $\Mman$,
\item\label{enum:X:proper}
or a \term{proper two-sided of codimension $1$}, that is
\begin{enumerate}[label={\rm(\alph*)}]
    \item $\partial\Yman = \Yman \cap \partial\Mman$ and that intersection is transversal;
    \item any tubular neighborhood $\Eman_{\Yman}\to\Yman$ of $\Yman$ is a trivial one-dimensional vector bundle.
\end{enumerate}
\end{enumerate}

Then the inclusion $\DiffNb{\Mman}{\Xman} \subset \DiffPlFix{\Mman}{\Xman}$ is a homotopy equivalence.
\end{subtheorem}
This statement might be deduced from ambient isotopy extension theorem for smooth submanifolds and contractibility of the space of tubular neighborhoods of a submanifold, see \cite[Chapter A, Proposition 31]{Godin:2008}.
A formal proof can be obtained as well from ``linearization theorem'' proved in the preprint~\cite{KhokhliukMaksymenko:2022}, see also discussions in~\cite[Theorem~2.2.5]{Maksymenko:DGA:2024}.
For the convenience of the reader let us present the arguments.
\begin{proof}
For simplicity assume that $\Xman$ is a closed submanifold of $\Int{\Mman}$: for boundary components one should consider collars (``halves'' of tubular neighborhoods).
Let $\pr\colon\Eman\to\Xman$ be any tubular neighborhood of $\Xman$ and $\VBAut[\Xman]{\Eman}$ be the group of $\Cinfty$ vector bundle automorphisms of $\Eman$ fixed on $\Xman$.
Denote by $\DiffFix{\Mman}{\Xman,\pr}$ the subgroup of $\DiffFix{\Mman}{\Xman}$ consisting of diffeomorphisms $\dif$ which coincide near $\Xman$ with some vector bundle automorphism of $\Eman$, which in fact is uniquely determined by the following formula:
\[
    \tfib{\dif}\colon\Eman\to\Eman,
    \qquad
    \tfib{\dif}(x) = \lim_{t\to 0} \tfrac{1}{t} \dif(t x).
\]
Here the multiplication by $t$ and $1/t$ should be understood in terms of the vector bundle structure on $\Eman$, see e.g.~\cite[Eq~(2.1)]{Maksymenko:DGA:2024}.

Let $\DiffFix{\Mman}{\Xman,\pr}$ be the subgroup of $\DiffFix{\Mman}{\Xman}$ consisting of diffeomorphisms $\dif$ which coincide near $\Xman$ with $\tfib{\dif}$.
Then under assumption~\ref{enum:X:proper} on $\Xman$, the ``linearization theorem'' from~\cite{KhokhliukMaksymenko:2022} claims that the inclusion $\DiffFix{\Mman}{\Xman,\pr} \subset  \DiffFix{\Mman}{\Xman}$ is a homotopy equivalence.

Denote $\DiffPlFix{\Mman}{\Xman,\pr} = \DiffPlFix{\Mman}{\Xman}\cap\DiffFix{\Mman}{\Xman,\pr}$,
Since $\DiffPlFix{\Mman}{\Xman}$ is a union of path components of $\DiffFix{\Mman}{\Xman}$, the inclusion
\begin{equation}\label{equ:DplFMXp__D+FMX}
    \DiffPlFix{\Mman}{\Xman,\pr} \subset \DiffPlFix{\Mman}{\Xman}
\end{equation}
is still a homotopy equivalence.

Furthermore, we have a homomorphism $\tfib{}\colon\DiffFix{\Mman}{\Xman,\pr}\to\VBAut[\Xman]{\Eman}$, $\dif\mapsto\tfib{\dif}$, whose kernel is $\DiffNb{\Mman}{\Xman}$.
It is shown in~\cite{KhokhliukMaksymenko:2022}, see also~\cite[Corollary~7.1.3]{Maksymenko:DGA:2024} for another proof, that $\tfib{}$ admits local cross sections, and therefore it is a locally trivial fibration \term{over its image} which is in turn a union of path components of $\VBAut[\Xman]{\Eman}$.

Now let us use the assumption that $\Xman$ is ``one-sided''.
Note that every vector bundle automorphism $\gdif$ of the trivial vector bundle $\Xman\times\bR\to\Xman$ is given by $\gdif(x,v) = (x, \alpha(x)v)$, for some everywhere non-zero $\Cinfty$ function $\alpha\colon\Xman\to\bR\setminus0$.
Hence $\VBAut[\Xman]{\Eman}$ can be identified with $\Ci{\Xman}{\bR\setminus0}$.

Moreover, then $\tfib{}(\DiffPlFix{\Mman}{\Xman,\pr}) \subset \Ci{\Xman}{(0;+\infty)}$.
Evidently, $\Ci{\Xman}{(0;+\infty)}$ is a path component of $\Ci{\Xman}{\bR\setminus0}$ containing the constant function $\Xman\mapsto 1$ being the image $\tfib{}(\id_{\Mman})$ of $\id_{\Mman}$.
Hence $\tfib{}(\DiffPlFix{\Mman}{\Xman,\pr}) = \Ci{\Xman}{(0;+\infty)}$, and the induced map
\[
    \tfib{}\colon\DiffPlFix{\Mman}{\Xman,\pr} \to \Ci{\Xman}{(0;+\infty)}
\]
is still a locally trivial fibration with fiber $\DiffNb{\Mman}{\Xman}$.
Since $\Ci{\Xman}{(0;+\infty)}$ is evidently contractible, we obtain that the inclusion of the fiber $\DiffNb{\Mman}{\Xman} \subset \DiffPlFix{\Mman}{\Xman,\pr}$ is a homotopy equivalence.
Together with the homotopy equivalence~\eqref{equ:DplFMXp__D+FMX} this gives the desired result.
\end{proof}

\subsection{Klein bottle}\label{sect:klein-bottle}
Let $\Torus = \Circle \times \Circle$ be the $2$-torus.
It will be convenient to say that the curves $\Circle\times\{w\}$, $w\in\Circle$, are \term{longitudes}, while the curves $\{w\}\times\Circle$ are \term{meridians}.

Consider the following orientation reversing involution without fixed points:
\[
    \xi\colon\Torus\to\Torus,
    \qquad
    \xi(z,w) = (-z, \bar{w}).
\]
It defines an action of $\bZ_2$ on $\Torus$, so that the corresponding quotient space $\KleinBottle:=\Torus/\xi$ is a non-orientable surface called \term{Klein bottle}.
In turn, the respective quotient map $\pr\colon\Torus\to\KleinBottle$ is then an oriented double covering of $\KleinBottle$.

We recall here several facts (more detailed than Table~\ref{tbl:hom_types:DidMX}) on the homotopy type of the group $\Diff(\KleinBottle)$ of diffeomorphisms of $\KleinBottle$ and the generators of $\pi_0\Diff(\KleinBottle)$.

\newcommand\tlongi[1]{\hat{\lambda}_{#1}}
\newcommand\tmerid[1]{\hat{\mu}_{#1}}
\newcommand\talpha{\tilde{\alpha}}
\newcommand\klongi[1]{\lambda_{#1}}
\newcommand\kmerid[1]{\mu_{#1}}

\newcommand\kyHom{y}
\newcommand\tyHom{\hat{\kyHom}}
\newcommand\ktHom{\tau}
\newcommand\ttHom{\hat{\ktHom}}

Let $J_{1}$ and $J_{2}$ be connected components of $\Circle\setminus\{\pm i\}$ containing $-1$ and $1$ respectively.
Consider the following union $\talpha = \Circle \times \{\pm i\}$ of two longitudes.
Then $\Torus\setminus\talpha$ consists of two connected components whose closures are the following cylinders:
\begin{align*}
    \tMob{1} &= \Circle\times J_{1}, &
    \tMob{2} &= \Circle\times j_{2}.
\end{align*}
Evidently, $\xi$ exchanges the longitudes $\Circle \times \{i\}$ and $\Circle \times \{-i\}$ and also leaves $\tMob{1}$ and $\tMob{2}$ invariant.
Denote $\alpha = \pr(\talpha)$, $\kMob{i} = \pr(\tMob{i})$, $i=1,2$.
Then $\kMob{1}$ and $\kMob{2}$ are M\"{o}bius bands in $\KleinBottle$ so that $\alpha = \partial\kMob{1} = \partial\kMob{2} = \kMob{1}\cap\kMob{2}$ and $\KleinBottle = \kMob{1} \cup \kMob{2}$.

Then $\klongi{1} = \pr(\Circle\times\{1\})$ and $\klongi{2} = \pr(\Circle\times\{-1\})$ are the ``\term{middle circles}'' of the M\"{o}bius bands $\kMob{1}$ and $\kMob{2}$ respectively, while $\kmerid{} = \pr(\{1\}\times\Circle)$ is the ``\term{meridian}'' of $\KleinBottle$.

Define further the following diffeomorphisms $\tyHom,\ttHom\colon\Torus\to\Torus$:
\[
    \tyHom(z,w)=(\bar{z},w),
    \qquad
    \ttHom(z,w)=(z,-\bar{w}).
\]
Evidently, they both have order $2$ and commute.
Moreover, they also commute with $\xi$, and therefore yield unique diffeomorphisms $\kyHom,\ktHom\colon\KleinBottle\to\KleinBottle$, such that $\pr\circ\tyHom=\kyHom\circ\pr$ and $\pr\circ\ttHom=\ktHom\circ\pr$.

\begin{subremark}\rm
Note that $\tyHom$ leaves invariant each longitude of $\Torus$ but reverses its orientation.
In particular, $\tMob{1}$ and $\tMob{2}$ are invariant under $\tyHom$, whence $\kyHom$ leaves invariant $\kMob{1}$ and $\kMob{2}$ but reverses orientation of their common boundary $\alpha$ and the middle circles $\klongi{1}$ and $\klongi{2}$.

On the other hand, $\ttHom$ pointwise fixes $\talpha$ and exchanges $\tMob{1}$ and $\tMob{2}$, whence $\ktHom$ pointwise fixes $\alpha$ and exchanges $\kMob{1}$ and $\kMob{2}$.

In fact, $\ktHom$ is isotopic to a Dehn twist along the meridian $\mu$, while $\kyHom$ is the \term{$Y$-homeomorphism} introduced by Lickorish~\cite{Lickorish:MPCPS:1963}.
\end{subremark}

\begin{subremark}\label{rem:curves_on_K}\rm
Let $\omega \subset \KleinBottle$ be a simple closed curve.
Then $\omega$ is called \term{two-sided} or \term{orientation preserving} if it has a neighborhood homeomorphic to an open cylinder.
Otherwise, $\omega$ has a neighborhood homeomorphic with an open M\"{o}bius band, and is called in this case \term{one-sided} or \term{orientation reversing}.
Fix some orientation of $\omega$, and denote by $-\omega$ the same curve but passed in the opposite direction.
Then by~\cite[Lemma~1]{Lickorish:MPCPS:1963} exactly one of the following cases is possible:
\begin{enumerate}[label={\rm(\alph*)}]
\item\label{enum:K_curve:mob_bands}
$\KleinBottle\setminus\omega$ is a union of two open M\"{o}bius bands, and in this case $\omega$ is isotopic to $\alpha$ or $-\alpha$;

\item\label{enum:K_curve:disk}
$\omega$ bounds a $2$-disk;

\item\label{enum:K_curve:meridian}
$\omega$ does not cut $\KleinBottle$, and is isotopic to the \term{meridian} $\kmerid{}$ of $\KleinBottle$.
\end{enumerate}
\end{subremark}

Define further the following $\Cinfty$ action of $\Circle$ on $\KleinBottle$ by
\begin{equation}\label{equ:S1_act_on_T2}
    \tSact\colon\Torus\times\Circle\to\Torus,
    \qquad
    \tSact(z,w,t) = (z t, w).
\end{equation}
Evidently,
\begin{equation}\label{equ:Gt_commute_xi_t_y}
\begin{aligned}
\tSact_t\circ \xi(z,w)
    &= \tSact_t(-z,\bar{w})
     = (-z t,\bar{w}) = \xi(z t, w) = \xi\circ \tSact_t(z,w),
\\
\tSact_t\circ \ttHom(z,w)
    &= \tSact_t(z,-\bar{w})
     = (z t, -\bar{w}) = \ttHom(z t, w) = \ttHom\circ \tSact_t(z,w),
\\
\tSact_t\circ \tyHom(z,w)
    &= \tSact_t(\bar{z},w)
     = (\bar{z} t, w) = \tyHom(z \bar{t}, w) = \tyHom\circ \tSact_{\bar{t}}(z,w).
\end{aligned}
\end{equation}
In other words, this action commutes with $\xi$ and therefore it yields a $\Cinfty$ action $\kSact_t\colon\KleinBottle\times\Circle\to\KleinBottle$ such that $\pr\circ \tSact_t = \kSact_t\circ\pr$.
Since $\bar{t} = t^{-1}$ for all $t\in\Circle$, the latter two identities in~\eqref{equ:Gt_commute_xi_t_y} imply that
\begin{align*}
    &\kSact_t \circ \ktHom = \ktHom \circ \kSact_t, &
    &\kSact_t \circ \kyHom = \kyHom \circ \kSact_t^{-1}.
\end{align*}
Hence we have the following subgroups of $\Diff(\KleinBottle)$:
\begin{align*}
&SO(2) \cong \langle\{\kSact_t\}_{t\in\Circle}\rangle,&
&O(2) \cong \langle SO(2),  \ \kyHom  \rangle,&
&O(2) \times \bZ_2 \cong \langle O(2), \ktHom \rangle.
\end{align*}

\begin{subtheorem}\label{th:DiffKleinBottle}
The inclusion $O(2) \times \bZ_2 \cong \langle \{\kSact_t\}_{t\in\Circle}, \kyHom, \ktHom \rangle \subset \Diff(\KleinBottle)$ is a homotopy equivalence.
In particular, $\DiffId(\KleinBottle)$ is homotopy equivalent to $SO(2)$, and $\pi_0\DiffId(\KleinBottle) \cong \langle\kyHom, \ktHom\rangle = \bZ_2\oplus\bZ_2$.

Let also $\DiffPlFix{\KleinBottle}{\alpha}$ be the group of diffeomorphisms of $\KleinBottle$ fixed on $\alpha$ and leaving invariant $\kMob{1}$ and $\kMob{2}$.
Then $\DiffPlFix{\KleinBottle}{\alpha}$ is contractible and, in particular, coincides with $\DiffIdFix{\KleinBottle}{\alpha}$.
\end{subtheorem}
\begin{proof}
The statement that the inclusion $SO(2) \subset \DiffId(\KleinBottle)$ is a homotopy equivalence is independently proved in~\cite[page 41, Corollary]{EarleEells:JGD:1969} and \cite[Th\'{e}or\`{e}me~1]{Gramain:ASENS:1973}.
The group $\pi_0\DiffId(\KleinBottle)$ and its generators $\kyHom$ and $\ktHom$ are described in~\cite{Lickorish:MPCPS:1963}.

To prove contractibility of $\DiffPlFix{\KleinBottle}{\alpha}$ note that by Theorem~\ref{th:DnbMX__DplFMX_he} the inclusions
\begin{align*}
    &\DiffNb{\KleinBottle}{\alpha}\subset\DiffPlFix{\KleinBottle}{\alpha}, &
    &\DiffNb{\kMob{i}}{\alpha} \subset \DiffFix{\kMob{i}}{\alpha}, \ i=1,2,
\end{align*}
are homotopy equivalences, and by Table~\ref{tbl:hom_types:DidMX} (first line) the groups $\DiffFix{\kMob{i}}{\alpha}$ are contractible.
Moreover, the map
\[
    \rho\colon\DiffNb{\KleinBottle}{\alpha} \to \DiffNb{\kMob{1}}{\alpha} \times \DiffNb{\kMob{2}}{\alpha},
    \qquad
    \rho(\dif) = \bigl(\restr{\dif}{\kMob{1}}, \restr{\dif}{\kMob{2}}\bigr),
\]
is, evidently, an \term{isomorphism} of topological groups.
Hence, $\DiffNb{\KleinBottle}{\alpha}$ and therefore $\DiffPlFix{\KleinBottle}{\alpha}$, are also contractible.
\end{proof}

\begin{subcorollary}\label{cor:isot_klein_bottle}
{\rm 1)}~The homomorphism $\mathsf{t}\colon\Diff(\KleinBottle) \to \Aut\bigl(H_1(\KleinBottle, \bZ)\bigr)$, associating to each $\dif\in\Diff(\KleinBottle)$ the corresponding automorphism of the first integral homology group $H_1(\KleinBottle, \bZ)$, is surjective and has the kernel $\DiffId(\KleinBottle)$.
In particular, it yields an isomorphism
\[
\pi_0\DiffId(\KleinBottle) \cong
\Diff(\KleinBottle)/\DiffId(\KleinBottle)\cong
\langle\ktHom, \kyHom\rangle = \bZ_2\oplus\bZ_2 \cong
\Aut\bigl(H_1(\KleinBottle, \bZ)\bigr).
\]

{\rm 2)}~Every diffeomorphism $\dif\in\Diff(\KleinBottle)$ is isotopic to a diffeomorphism $\dif'$ leaving invariant $\alpha$, and thus preserving or exchanging $\kMob{1}$ and $\kMob{2}$.
In that case
\begin{itemize}
\item $\dif'$ preserves $\kMob{1}$ and $\kMob{2}$ iff it is isotopic either to $\id_{\KleinBottle}$ or to $\kyHom$;
\item $\dif'$ preserves orientation of $\alpha$ iff it is isotopic either to $\id_{\KleinBottle}$ or to $\ktHom$.
\end{itemize}

{\rm 3)}~Let $\gamma\colon[0;1]\to\Diff(\KleinBottle)$ be a loop at $\id_{\KleinBottle}$, i.e.\ $\gamma$ is a continuous path with $\gamma(0)=\gamma(1)=\id_{\KleinBottle}$.
Fix any point $u\in\KleinBottle$ and consider the loop $\gamma^u\colon[0;1]\to\KleinBottle$, $\gamma^u(t) = \gamma(t)(u)$ being the ``\term{trace}'' of $u$.
Then the path $\gamma$ is a generator of $\pi_1\DiffId(\KleinBottle,\id_{\KleinBottle})\cong \bZ$ iff $\gamma^u=\pm\alpha$ in $H_1(\KleinBottle,\bZ)$.
\end{subcorollary}
\begin{proof}
1) Choose the standard counterclockwise orientation of $\Circle$.
This gives orientations on $\Circle\times\{i\}$, $\Circle\times\{1\}$, and $\{1\}\times\Circle$.
Choose also on $\Circle\times\{-i\}$ and $\Circle\times\{-1\}$ opposite orientations.
This gives orientations on $\klongi{1}$, $\klongi{2}$, $\alpha$ and $\kmerid{}$, so they can be regarded as elements of $H_1(\KleinBottle, \bZ)$.
Then $H_1(\KleinBottle, \bZ) \cong \bZ \oplus \bZ_2 = \langle\klongi{1}\rangle \oplus \langle\kmerid{}\rangle$.
Thus with these coordinates
\begin{align*}
    &\klongi{1} = (1,0),&
    &\kmerid{} = (0,1),&
    &\klongi{2} = (1,1),&
    &\alpha = \klongi{1} = \klongi{2} = (2,0).
\end{align*}
It is easy to see that $\Aut(H_1(\KleinBottle, \bZ))$ isomorphic with $\bZ_{2}\oplus\bZ_{2}$ and consists of the following automorphisms begin exactly the images of distinct elements of $\langle\kyHom, \ktHom\rangle$:
\begin{equation}\label{equ:aut_H1K}
\begin{aligned}
\mathsf{t}(\id_{\KleinBottle}) &\colon (n,\delta) \mapsto (n,  \delta), &\qquad\qquad
\mathsf{t}(\kyHom)             &\colon (n,\delta) \mapsto (-n, \delta), \\
\mathsf{t}(\ktHom)             &\colon (n,\delta) \mapsto (n,  \delta + (n\bmod2)), &
\mathsf{t}(\ktHom\circ\kyHom)  &\colon (n,\delta) \mapsto (-n, \delta + (n\bmod2)),
\end{aligned}
\end{equation}
where $n\in\bZ$ and $\delta\in\bZ_{2} = \{0,1\}$.

2) By Theorem~\ref{th:DiffKleinBottle} every $\dif\in\Diff(\KleinBottle)$ is isotopic to some diffeomorphism  from the group $\langle\kyHom, \ktHom\rangle$, each of whose elements leave $\alpha$ invariant.

Now let $\dif\in\Diff(\KleinBottle)$ be such that $\dif(\alpha)=\alpha$.
Then $\dif$ preserves $\kMob{1}$ and $\kMob{2}$ iff $\mathsf{t}(\dif)$ preserves $\klongi{1}$ and $\klongi{2}$, so it coincides with either automorphisms from the first line of~\eqref{equ:aut_H1K}, and thus $\dif$ is isotopic either to $\id_{\KleinBottle}$ or to $\kyHom$.

Similarly, $\dif$ preserves orientation of $\alpha$ iff $\mathsf{t}(\dif)(\alpha) = \alpha$, that is $\mathsf{t}(\dif)(2,0) = (2,0)$.
But this means that $\mathsf{t}(\dif)(1,0)$ equals either to $(1,0)$ or $(1,1)$, whence $\mathsf{t}(\dif)$ coincides with either of automorphisms from the left column of~\eqref{equ:aut_H1K}, and thus $\dif$ is isotopic either to $\id_{\KleinBottle}$ or to $\ktHom$.

3) First note that since $\KleinBottle$ is path connected, the \term{homotopy class of $\gamma^u$ does not depend on a particular choice of $u\in\KleinBottle$}.
Indeed, evidently, $\gamma$ yields an isotopy $G\colon[0;1]\times\KleinBottle\to\KleinBottle$ given by $G(t,x) =\gamma(t)(x)$.
In particular, $\gamma^u = \restr{G}{[0;1]\times\{u\}}$.
Let $v\in\KleinBottle$ be any point, and $\phi\colon[0;1]\to\KleinBottle$ be any path with $\phi(0)=u$ and $\phi(1)=v$.
Then the map $\Phi\colon[0;1]\times[0;1]\to\KleinBottle$, $\Phi(s,t)=G(t,\phi(s))$, is a homotopy between the loops $\gamma^u$ and $\gamma^v$.

Let us also show that \term{the homotopy class of $\gamma^u$ does not change under homotopies of $\gamma$.}
Indeed, let $h\colon[0;1]\times[0;1]\to\Diff(\KleinBottle)$ be any homotopy of $\gamma$, i.e.\ $h_0 = \gamma$ and $h([0;1]\times\{0,1\}) = \{\id_{\KleinBottle}\}$.
Then it also yields a ``homotopy of isotopies'' $H\colon[0;1]\times[0;1]\times\KleinBottle\to\KleinBottle$, $H(s,t,x) = h(s,t,x)$, so that its restriction $h^u = \restr{G}{[0;1]\times[0;1]\times\{u\}}$ is a homotopy of $\gamma^u$.
In particular, for each $t\in[0;1]$ the homotopy class of the loop $(h_t)^u = \restr{G}{\{t\}\times[0;1]\times\{u\}}$ coincides with $(h_0)^u=\gamma^u$.

Finally, since the inclusion $SO(2) \subset \DiffId(\KleinBottle)$ is a homotopy equivalence, and due to the previous two paragraphs, one can assume that $\gamma([0;1]) \subset SO(2) = \{\kSact_{t}\}_{t\in\Circle}$ and $u\in\alpha$.
Note also that the map $\rho_{u}\colon SO(2) \to \alpha$, $\rho_{u}(t)=\kSact_{t}(u)$, is a homeomorphism, and $\gamma^u = \rho_{u} \circ \gamma$.
Hence, $\gamma^u$ represents a loop in $\KleinBottle$ homotopic to $\alpha$ iff $\gamma$ represents a generator of $\pi_1 SO(2) = \pi_1\DiffId(\KleinBottle)$.
\end{proof}


\section{Proof of Lemma~\ref{intro:lm:top_decomposition}}\label{sect:proof:prop:top_decomposition}
Let $\func\in\fclassr{\KleinBottle}$.
We need to prove that exactly one of the cases~\ref{intro:klein_topmebius}--\ref{into:graph:one_cycle} of Lemma~\ref{intro:lm:top_decomposition} holds.

Suppose $\krgraphf$ is a tree, i.e.\ has no cycles.
This means that every point of every open edge of $\krgraphf$ splits that graph into two pieces.
In other words, every regular contour $\alpha$ of $\func$ (as a point on some open edge of $\krgraphf$) separates $\KleinBottle$ into two connected subsurfaces $\kMob{1}$ and $\kMob{2}$.
As mentioned in Remark~\ref{rem:curves_on_K}, this implies that
\begin{itemize}
   \item[$(*)$] either $\kMob{1}$ and $\kMob{2}$ are M\"{o}bius bands, so we get the case~\ref{intro:klein_topmebius};
   \item[$(**)$] or $\kMob{1}$ is a Klein bottle with hole and $\kMob{2}$ is an open $2$-disk.
\end{itemize}

Thus assume that \term{each regular contour of $\func$ satisfies $(**)$, i.e.\ bounds a $2$-disk}.
We need to show that then the case~\ref{intro:klein_topdisks} holds, i.e.\ there exists a unique \textsl{critical} contour $\beta$ of $\func$ such that the complement $\KleinBottle\setminus\beta$ is a disjoint union of finitely many open disks.

\newcommand\uvrt{u}
\newcommand\vvrt{v}
\newcommand\wvrt{w}
\begin{lemma}\label{lm:krgraph-orientation}
If each regular contour of $\func$ bounds a $2$-disk, then $\krgraphf$ admits an orientation of edges such that each of its vertices has at most one incoming edge.
\end{lemma}
\begin{proof}
Let $(\uvrt,\vvrt)$ be an open edge of $\krgraphf$ and $\wvrt\in(\uvrt,\vvrt)$ be any point, so $\uvrt$ and $\vvrt$ are critical contours of $\func$, while $\wvrt$ is a regular contour of $\func$.
By assumption $\KleinBottle\setminus\wvrt = K_{\wvrt} \sqcup D_{\wvrt}$, where $D_{\wvrt}$ is an open $2$-disk and $K_{\wvrt}$ is a Klein bottle with hole.
Then $\uvrt$ and $\vvrt$ are contained in distinct connected components of $\KleinBottle\setminus\wvrt$, and we orient the edge $(\uvrt,\vvrt)$ in the ``direction of $2$-disk'', i.e.\ orient $(\uvrt,\vvrt)$ from $\uvrt$ to $\vvrt$ iff either of the following equivalent conditions hold: $\uvrt\subset K_{\wvrt}$ or $\vvrt \subset D_{\wvrt}$.
In this case we will write $\uvrt <_{\wvrt} \vvrt$.

i) Let us show that \term{this definition does not depend on a particular choice of $\wvrt\in(\uvrt,\vvrt)$}.
We need to show that for any other point $\wvrt'\in(\uvrt,\vvrt)$ the conditions $\uvrt <_{\wvrt} \vvrt$ and $\uvrt <_{\wvrt'} \vvrt$ are equivalent.
Exchanging $\wvrt$ and $\wvrt'$ if necessary one can assume that $\wvrt\in(\uvrt,\wvrt')$.

Notice that $\wvrt$ and $\wvrt'$ bound an open cylinder $C$.
Let also $\KleinBottle\setminus\wvrt' = K_{\wvrt'} \sqcup D_{\wvrt'}$ be the corresponding splitting of $\KleinBottle$.
Since a Klein bottle with hole is not embeddable into a $2$-disk, it follows that if $\wvrt'\in(\uvrt,\wvrt)$ and $\uvrt <_{\wvrt} \vvrt$, i.e.\ $\uvrt \subset K_{\wvrt}$, then
\begin{align*}
    &\uvrt \subset K_{\wvrt} = K_{\wvrt'} \cup \wvrt' \cup C, &
    &D_{\wvrt'} = C \cup \wvrt \cup D_{\wvrt} \supset \vvrt.
\end{align*}
The latter inclusion $\vvrt \subset D_{\wvrt'}$ means that $\uvrt <_{\wvrt'} \vvrt$.
Hence if $\vvrt <_{\wvrt} \uvrt$, then $\vvrt <_{\wvrt'} \uvrt$ as well.

Thus, we can now write $\uvrt < \vvrt$ instead of $\uvrt <_{\wvrt} \vvrt$.

ii) Let us prove that every vertex $\vvrt$ of $\krgraphf$ has at most one incoming edge.
Suppose $(\uvrt,\vvrt)$ and $(\uvrt',\vvrt)$ are two incoming open edges for vertex $\vvrt$, $\wvrt\in(\uvrt,\vvrt)$ and $\wvrt'\in(\uvrt',\vvrt)$ be two points, and $\KleinBottle\setminus\wvrt = K_{\wvrt} \sqcup D_{\wvrt}$ and $\KleinBottle\setminus\wvrt' = K_{\wvrt'} \sqcup D_{\wvrt'}$ be the corresponding splitting of the complements of $\wvrt$ and $\wvrt'$.
The assumption $\uvrt<\vvrt$ and $\uvrt'<\vvrt$ means that $K_{\wvrt}$ and $K_{\wvrt'}$ are disjoint subsurfaces of the Klein bottle each homeomorphic to a Klein bottle with hole, which is impossible.
Hence $\vvrt$ may have at most one incoming edge.
\end{proof}

\begin{corollary}
Let $\Gamma$ be a finite directed tree such that for every vertex there exists at most one incoming edge.
Then $\Gamma$ has a unique ``minimal'' vertex $\beta$, i.e.\ all edges incident to $\beta$ are outgoing.
\end{corollary}
\begin{proof}
First let us show that such a vertex exists.
Indeed, let $\uvrt_0$ be any vertex of $\Gamma$.
If it is not minimal, then by assumption there exist a unique incoming edge $(\uvrt_1, \uvrt_0)$.
In turn, if $\uvrt_1$ is not minimal, there exist a unique incoming edge $(\uvrt_2, \uvrt_1)$ and so on.
Since $\Gamma$ is connected and contains no cycles, that process will stop at some vertex $\uvrt_n$ with all outgoing edges.

Now let us show that such a vertex is unique.
Suppose $\beta$ and $\beta'$ are two distinct minimal vertices.
Since $\Gamma$ is a tree, there exists a unique path $\beta=\uvrt_0, \uvrt_{1},\ldots,\uvrt_{n-1},\uvrt_{n}=\beta'$.
One can also assume that this path contains no other minimal vertices.
As $\beta$ is minimal, we have that $\beta =\uvrt_0 < \uvrt_1$, so $(\uvrt_0,\uvrt_1)$ is an incoming edge for $\uvrt_1$.
But then such an edge is unique for $\uvrt_1$, and thus $\uvrt_1 < \uvrt_2$.
Similarly, one proves that $\uvrt_0 < \uvrt_1 < \uvrt_2 <\cdots < \uvrt_{n-1} < \uvrt_{n}=\beta'$.
In particular, $\uvrt_{n-1} <\beta'$ which contradicts to the assumption that $\beta'$ is also minimal.
\end{proof}

It remains to note that the minimal vertex $\beta$ of $\krgraphf$ is exactly the critical contour satisfying~\ref{intro:klein_topdisks}.

\ref{into:graph:one_cycle}
Suppose $\krgraphf$ has a cycle $\Omega$.
Take any point $\gamma$ on some open edge of the cycle $\Omega$, so $\gamma$ is a regular contour of $\func$.
For simplicity assume that $\func(\gamma)$ is a regular value of $\func$, that is the corresponding level set $\Qman := \func^{-1}\bigl(\func(\gamma)\bigr)$ of $\func$ contains no critical points of $\func$.
Then $\Qman$ is a compact $1$-dimensional manifold of $\func$ without boundary, and thus it is a union of finitely many contours of $\func$.

We claim that $\gamma$ satisfies statement~\ref{into:graph:one_cycle} of Lemma~\ref{intro:lm:top_decomposition}.
Indeed, by definition $\dif(\Qman) = \Qman$ for all $\dif\in\Stabilizer{\func}$.
In particular, $\{\dif(\gamma)\}_{\dif\in\Stabilizer{\func}} \subset \Qman$ is a disjoint union of finitely many regular contours $\{\gamma_1,\dots,\gamma_m\}$.
Since each of them is isotopic to the meridian, it follows that $\KleinBottle\setminus \{\gamma_1,\dots,\gamma_m\}$ is a disjoint union of $m$ open cylinders, see e.g.~\cite[Lemma~2.4.(i)]{Epstein:AM:1966}.

\section{Proof of Theorem~\ref{intro:th:acylcic_mobius_orbit}}
\label{sect:proof:intro:th:acylcic_mobius_orbit}
Let $\func\in\fclassr{\Mman}$ be a function satisfying the statement~\ref{intro:klein_topmebius} or Lemma~\ref{intro:lm:top_decomposition}, i.e.\ there exists a \term{regular} contour $\alpha$ of $\func$ such that the complement $\KleinBottle\setminus\alpha$ is a disjoint union of two open M\"{o}bius bands.
Denote their closures by $\kMob{1}$ and $\kMob{2}$.
Then $\alpha = \partial\kMob{1} = \partial\kMob{2} = \kMob{1}\cap\kMob{2}$.

Let $\DiffPlInv{\KleinBottle}{\alpha}$ be the group of diffeomorphisms $\dif$ of $\KleinBottle$ which leave both $\kMob{1}$ and $\kMob{2}$ invariant, and
$\DiffPlFix{\KleinBottle}{\alpha}$ be the subgroup of $\DiffPlInv{\KleinBottle}{\alpha}$ fixed on $\alpha$.

Before proving Theorem~\ref{intro:th:acylcic_mobius_orbit}, let us establish certain properties of stabilizers and orbits of $\func$ with respect to diffeomorphism groups related with $\alpha$.
Recall that we have the action map~\eqref{intro:equ:action}:
\begin{equation}\label{equ:p_DK_Of}
    \actmap\colon\DiffId(\KleinBottle)\to\OrbitComp{\func}{\func},
    \qquad
    \actmap(\dif)=\func\circ\dif,
\end{equation}

\begin{lemma}\label{lm:prop:Da}
The following statements hold.
\begin{enumerate}[label={\rm(\arabic*)}, itemsep=1ex, leftmargin=*]
\item\label{enum:lm:prop:Da:Circle}
The inclusion $j\colon SO(2)\subset \DiffPl(\Circle)$ is a homotopy equivalence, e.g.~\cite[Prop.~4.2]{Ghys:EM:2001}.

\item\label{enum:lm:prop:Da:Sf_DplKa}
$\Stabilizer{\func} \subset \DiffPlInv{\KleinBottle}{\alpha}$, i.e.\ every $\dif\in\Stabilizer{\func}$ leaves invariant $\alpha$, and therefore each $\kMob{1}$ and $\kMob{2}$.
In particular, by Corollary~\ref{cor:isot_klein_bottle}, $\dif$ is isotopic
\begin{itemize}
\item either to $\id_{\KleinBottle}$, i.e.\ belongs to $\StabilizerIsotId{\func}$, and in that case it preserves orientation of $\alpha$;
\item or to the diffeomorphism $\kyHom\colon\KleinBottle\to\KleinBottle$, and then $\dif$ reverses orientation of $\alpha$.
\end{itemize}
Hence, either $\Stabilizer{\func}=\StabilizerIsotId{\func}$ or $\Stabilizer{\func}/\StabilizerIsotId{\func} \cong \bZ_2$.

\item\label{enum:lm:prop:Da:pik_S_Sa}
The ``restriction to $\alpha$ map'' $\rho_{\alpha}\colon\Stabilizer{\func} \to \Diff(\alpha)$, $\rho_{\alpha}(\dif) = \restr{\dif}{\alpha}$, is a locally trivial fibration with fiber $\Stabilizer{\func,\alpha}$.
In particular, $\rho_{\alpha}(\StabilizerIsotId{\func}) = \DiffPl(\alpha)$ and
we also have isomorphisms
\begin{equation}\label{equ:pi_k_S_Sa}
   \pi_k\bigl( \StabilizerIsotId{\func}, \Stabilizer{\func,\alpha}\bigr) \,\cong\,
   \pi_k\DiffPl(\alpha) \,\cong\,
   \pi_k\alpha =
   \begin{cases}
    \bZ, &k=1,\\
    0, &k\not=1
   \end{cases}
\end{equation}
where the first one is induced by $\rho_{\alpha}$ and comes from the long exact sequence of homotopy groups of $\rho_{\alpha}$, and the second one is induced by the inverse of the homotopy equivalence $j$ from~\ref{enum:lm:prop:Da:Circle}.

\item\label{enum:lm:prop:Da:Spla}
$\StabilizerIsotId{\func,\alpha} =
\Stabilizer{\func,\alpha}$.

\item\label{enum:lm:prop:Da:Opla}
$\OrbitPl{\func,\alpha} = \OrbitPlComp{\func}{\func,\alpha}= \OrbitComp{\func}{\func,\alpha}$.
\end{enumerate}
\end{lemma}
\begin{proof}


Statement~\ref{enum:lm:prop:Da:Circle} is well known and elementary.

\ref{enum:lm:prop:Da:Sf_DplKa}
Let $c = \func(\alpha)$ be the value of $\func$ on $\alpha$, $\Uman$ be a $\func$-regular neighborhood of $\alpha$.
Denote also $\Uman_i = (\Mobius_i\cap\Uman)\setminus\alpha$, $i=1,2$.
Then, not loosing generality, one can assume that $\func(\Uman_1) \subset (-\infty;c)$ and $\func(\Uman_2) \subset (c;+\infty)$.

Now let $\dif\in\Stabilizer{\func}$.
Then $\dif(\alpha) = \alpha$ and therefore $\dif(\Uman) = \Uman$.
Suppose $\dif$ exchanges $\kMob{1}$ and $\kMob{2}$, then it must also exchange $\Uman_1$ and $\Uman_2$.
However, since $\func\circ\dif=\func$, we will obtain that
\[
    (-\infty;c)             \ \supset \
    \func(\Uman_1)          \ = \
    \func\circ\dif(\Uman_1) \ = \
    \func(\Uman_2)          \ \subset \
    (c;+\infty)
\]
which is impossible.
Hence $\dif\in\DiffPlInv{\KleinBottle}{\alpha}$.

All other statements are trivial.

\ref{enum:lm:prop:Da:pik_S_Sa}
The fact that $\rho_{\alpha}$ is a locally trivial fibration is a particular case of~\cite[Theorem~8.2]{KhokhliukMaksymenko:IndM:2020}.

\ref{enum:lm:prop:Da:Spla}
Recal that by Theorem~\ref{th:DiffKleinBottle}, $\DiffIdFix{\KleinBottle}{\alpha} = \DiffPlFix{\KleinBottle}{\alpha}$ and this group is contractible.
Hence
\begin{align*}
    \StabilizerIsotId{\func,\alpha} :&\!=
    \Stabilizer{\func} \cap \DiffIdFix{\KleinBottle}{\alpha}
    \stackrel{\text{Th.\,\ref{th:DiffKleinBottle}}}{=\!=}
    \Stabilizer{\func} \cap \DiffPlFix{\KleinBottle}{\alpha}
    =
         \Stabilizer{\func} \cap \bigl(\DiffPlInv{\KleinBottle}{\alpha}
            \cap
        \DiffFix{\KleinBottle}{\alpha} \bigr) \\
    &\!=
         \bigl(\Stabilizer{\func} \cap \DiffPlInv{\KleinBottle}{\alpha} \bigr)
            \cap
        \DiffFix{\KleinBottle}{\alpha}
     \stackrel{\ref{enum:lm:prop:Da:Sf_DplKa}}{=}
        \Stabilizer{\func} \cap \DiffFix{\KleinBottle}{\alpha} =\Stabilizer{\func,\alpha}.
\end{align*}



\ref{enum:lm:prop:Da:Opla}
Note that
\[
    \OrbitComp{\func}{\func,\alpha} = p(\DiffIdFix{\KleinBottle}{\alpha})
    \stackrel{\text{Th.\,\ref{th:DiffKleinBottle}}}{=\!=}
    p(\DiffPlFix{\KleinBottle}{\alpha}) =
    \OrbitPl{\func,\alpha}.
\]
Moreover, since $\DiffPlFix{\KleinBottle}{\alpha}$ is path connected, $\OrbitPl{\func,\alpha}$ is path connected as well, and therefore it coincides with $\OrbitPlComp{\func}{\func,\alpha}$.
\end{proof}

To simplify notations we further put:
\begin{align*}
\xS &:= \StabilizerIsotId{\func},&
\xD &:= \DiffId(\KleinBottle),&
\xO &:= \OrbitComp{\func}{\func},\\
\xSa &:= \StabilizerIsotId{\func,\alpha}, &
\xDa &:= \DiffIdFix{\KleinBottle}{\alpha}, &
\xOa &:= \OrbitComp{\func}{\func,\alpha}.
\end{align*}
\begin{subremark}\label{rem:interpr_pi1_D_S1}\rm
Let $\pu\in\alpha$ be any point.
We will discuss now the isomorphisms~\eqref{equ:pi_k_S_Sa} for $k=1$:
\[
    \pi_1(\xS,\xSa,\id_{\KleinBottle}) \xrightarrow{~\rho_{\alpha}~}
    \pi_1\bigl(\DiffPl(\alpha),\id_{\Circle}\bigr)  \xrightarrow{~j^{-1}~}
    \pi_1(\alpha, \pu) \cong \bZ.
\]

Let $\gamma\colon[0;1]\to\DiffPl(\alpha)$ be loop at $\id_{\Circle}$, i.e.\ $\gamma(0)=\gamma(1)=\id_{\Circle}$.
Then $j^{-1}[\gamma]$ is represented by the loop $\gamma^{\pu}\colon [0;1]\to\alpha$, $\gamma^{\pu}(t) = \gamma(t)(\pu)$, so it associates to $\gamma$ the ``\term{trace}`` of the point $\pu$.

Moreover, since $\rho_{\alpha}$ is the restriction to $\alpha$ map, it follows that the composition of isomorphisms $\rho_{\alpha}\circ j^{-1}$ has a similar description.
Namely, let $\delta\colon\bigl([0;1], 1, 0\bigr) \to(\xS,\xSa,\id_{\KleinBottle})$ be a continuous map.
Then $\rho_{\alpha}\circ j^{-1}[\delta]$ is represented by the loop $\delta^{\pu}\colon [0;1]\to\xS$, $\delta^{\pu}(t) = \delta(t)(\pu)$, so it again associates to $\delta$ the ``\term{trace}`` of the point $\pu$.

In particular, if $\pu$ makes one revolution along $\alpha$ under $\delta$, then $\delta$ represents a generator of $\pi_1\pi_1(\xS,\xSa)$.
\end{subremark}

The first part of Theorem~\ref{intro:th:acylcic_mobius_orbit} is contained in the following lemma:
\begin{lemma}\label{lm:inc_Oa_O}
The inclusion $j\colon\xOa \subset \xO$ is a homotopy equivalence.
\end{lemma}
\begin{proof}
1) First we show that in order to prove the lemma it suffices to verify that the inclusion of pairs $\incDaSa\colon(\xDa,\xSa) \subset (\xD,\xS)$ yields an isomorphism of the relative fundamental groups:
\begin{equation}\label{equ:j1_DaSa_DS_iso}
    \incDaSa\colon \pi_1(\xDa,\xSa) \to \pi_1(\xD,\xS).
\end{equation}

Indeed, by statements \ref{enum:th:diff_act:Of_Freshet_manif} and \ref{enum:th:diff_act:DO_fibr} of Theorem~\ref{th:diff_act} the map~\eqref{equ:p_DK_Of} as well as its restriction $\restr{\actmap}{\xDa}\colon\xDa\to\xOa$ are locally trivial fibrations with fibers $\xS$ and $\xSa$ respectively.
In particular, $\actmap$ also induces isomorphisms of homotopy groups  $\pi_k(\xD,\xS) \cong \pi_k\xO$ and $\pi_k(\xDa,\xSa) \cong \pi_k\xOa$ for all $k\geq1$.

Since $\xD$ is homotopy equivalent to the circle (Theorem~\ref{th:DiffKleinBottle}), and $\xDa$ as well as connected components of $\xS$ and $\xSa$ are contractible (Theorem~\ref{th:DiffKleinBottle} and Theorem~\ref{th:hom_type:stabilizers}), it easily follows from exact sequences of the corresponding homotopy groups of those fibrations that $\xOa$ and $\xO$ are aspherical.

Moreover, again by statement \ref{enum:th:diff_act:Of_Freshet_manif} of Theorem~\ref{th:diff_act}, $\xOa$ and $\xO$ are homotopy equivalent to some CW-complexes, whence their homotopy types are determined by the corresponding fundamental groups.

Thus we just need to prove that the inclusion $j$ yields an isomorphism $j_1\colon\pi_1\xOa \cong \pi_1\xO$.
Finally, replacing these fundamental groups with the corresponding relative groups we are reduced to the proof that $\incDaSa$ is an isomorphism.

2) The following lemma allows to reduce the problem further.
\begin{sublemma}
We have the following commutative diagram:
\[
\xymatrix@R=1.5em{
    {\underbrace{\pi_1\xS}_{0}}
        \ar[r] &
    {\underbrace{\pi_1(\xS,\xSa)}_{\pi_1\DiffPl(\alpha) = \bZ}}
        \ar[r]^-{\bdSSa} \ar@{-->}[dd]_-{\eta}^-{\cong\,?} &
    \pi_0\xSa
        \ar[r]^-{\incSa} &
    \pi_0\xS
        \ar@{=}[dd] \ar[r] &
    {\underbrace{\pi_0(\xS,\xSa)}_{\pi_0\DiffPl(\alpha) = 0}}  \\
        & & \pi_1(\xDa,\xSa) \ar[u]^-{\bdDaSa}_-{\cong} \ar[d]_-{\incDaSa}^-{\cong\,?} &  \\
    {\overbrace{\pi_1\xS}^{0}} \ar[r] &
        {\overbrace{\pi_1\xD}^{\bZ}} \ar[r]^-{\fogD} &
        \pi_1(\xD,\xS) \ar[r]^-{\bdDS} &
        \pi_0 \xS \ar[r] &
    {\overbrace{\pi_0\xD}^{0}}
}
\]
in which the upper and lower rows are parts of exact sequences of the pairs $(\xS,\xSa)$ and $(\xD,\xS)$ respectively.
\end{sublemma}
\begin{proof}
Note note that the right square of that diagram is commutative, i.e.
\begin{equation}\label{equ:id_di}
    \bdDS\circ\incDaSa = \incSa\circ\bdDaSa.
\end{equation}
Indeed, take any representative $\gamma\colon[0;1]\to\xDa$ of some element of $\pi_1(\xDa,\xSa)$, so $\gamma(0)=\id_{\KleinBottle}$ and $\gamma(1)\in\xSa$.
Then it is evident, that both sides of~\eqref{equ:id_di} associate to $\gamma$ the isotopy class of $\gamma(1)$ regarded as an element of $\xS$.

Now, since $\xDa$ is contractible, $\bdDaSa$ is an isomorphism, whence the middle vertical arrow $\incDaSa\circ\bdDaSa^{-1}$ is well-defined.
Therefore it sends $\ker(\incSa)$ into $\ker(\bdDS)$, and thus induces a unique homomorphism $\eta\colon\pi_1(\xS,\xSa)\to\pi_1\xD$.
\end{proof}

We will show that \term{$\eta\colon\pi_1(\xS,\xSa)\to\pi_1\xD$ is an isomorphism}.
Then the Five Lemma will guarantee that $\incDaSa\circ\bdDaSa^{-1}\colon\pi_0\xSa\to \pi_1(\xD,\xS)$ is an isomorphism as well.
Hence $\incDaSa$ will also be an isomorphism, which will finish the proof of Lemma~\ref{lm:inc_Oa_O}.

\begin{figure}[htbp]
\includegraphics[height=4cm]{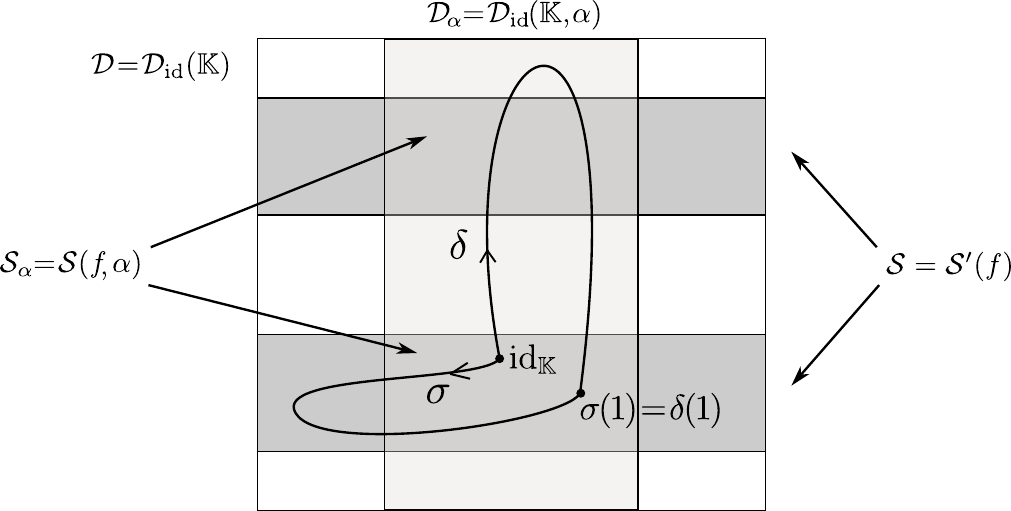}
\caption{Geometrical meaning of $\eta$: $\eta(\genSSa) = \genInv\#\genSSa^{-1}$}
\label{fig:eta}
\end{figure}

3) First let us describe the geometrical meaning of homomorphism $\eta$, see Figure~\ref{fig:eta}.

Let $\genSSa\colon\bigl([0;1],1,0\bigr)\to\bigl(\xS,\xSa,\id_{\KleinBottle}\bigr)$ be a representative of some element of $\pi_1(\xS,\xSa)$.
Then $\bdSSa(\genSSa) = [\genSSa(1)] \in\pi_0\xSa$ is a path component of the end $\genSSa(1)$ of that path in $\xSa$.

Since $\xSa \subset \xDa$ and $\xDa$ is contractible, there exists another path $\genInv\colon[0;1]\to\xDa$ such that $\genInv(0)=\id_{\KleinBottle}$ and $\genInv(1)=\genSSa(1)$.
In other words, $\genInv$ represents some element of $\pi_1(\xDa,\xSa)$,
As $\bdDaSa$ is an isomorphism, that element must be $\bdDaSa^{-1}\circ\bdSSa(\genSSa) \in \pi_1(\xDa,\xSa)$.

Taking further the composition of $\genInv$ with the inclusion $\xDa\subset\xD$, we see that $\genInv$ can be regarded as a path $\genInv\colon\bigl([0;1],1,0\bigr)\to\bigl(\xD,\xS,\id_{\KleinBottle}\bigr)$ representing the element
\begin{equation}\label{equ:genInv_in_pi1DS}
    \incDaSa\circ\bdDaSa^{-1}\circ\bdSSa(\genSSa) \in \pi_1(\xD,\xS).
\end{equation}

Consider the composition of paths: $\genD=\genInv\#\genSSa^{-1}\colon\bigl([0;1],1,0\bigr) \to \bigl(\xD,\xS,\id_{\KleinBottle}\bigr)$.
Since the image of $\genSSa$ is contained in $\xS$, $\genD$ represents the same element of $\pi_1(\xD,\xS)$ as~\eqref{equ:genInv_in_pi1DS}.

On the other hand, since the ends of $\genSSa$ and $\genInv$ coincide, $\genD=\genInv\#\genSSa^{-1}$ is in fact a loop in $\xD$ at $\id_{\KleinBottle}$, such that $\fogD([\genInv\#\genSSa^{-1}])=[\genD]\in\pi_1(\xD,\xS)$.

As $\fogD$ is injective, we should have that $\eta(\genSSa) = [\genInv\#\genSSa^{-1}]$.

4) It remains to show that $\eta$ sends some generator of $\pi_1(\xS,\xSa)\cong\bZ$ to generator of $\pi_1\xD\cong \bZ$.

Let $\Uman$ be a $\func$-regular neighborhood of $\alpha$.
Since $\alpha$ is a regular contour of $\func$, there exists an embedding $\phi\colon\Circle\times[-1;1]\to\KleinBottle$ such that $\phi(\Circle\times\{s\})$, $s\in[-1;1]$, is a regular contour of $\func$, and $\alpha = \phi(\Circle\times\{0\})$.

Let $\theta\colon[-1;1]\to[0;1]$ be a $\Cinfty$ function such that $\theta(0)=1$ and $\theta=0$ on some neighborhood of the boundary $\{-1, 1\}$.
Define the isotopy
\[
    G\colon\Circle\times[-1;1]\times[0;1]\to\Circle\times[-1;1],
    \qquad
    G(z,s,t) = (ze^{2\pi i t \theta(s)}, s).
\]
Evidently, it has the following properties:
\begin{enumerate}[label={\rm(\alph*)}]
   \item\label{enum:Gprop:fix_bd} $G$ is fixed near $\Circle\times\{\pm1\}$;
   \item\label{enum:Gprop:inv_circles} $G$ leaves invariant the circles $\Circle\times\{s\}$, $s\in[-1;1]$;
   \item\label{enum:Gprop:full_rot_0} $G$ makes a full revolution of $\Circle\times\{0\}$.
\end{enumerate}
Property~\ref{enum:Gprop:fix_bd} allows to define the following isotopy of $\KleinBottle$:
\[
    \Sigma\colon\KleinBottle\times[0;1]\to\KleinBottle,
    \qquad
    \Sigma(x,t) =
    \begin{cases}
        x, & x \in\KleinBottle\setminus\Uman, \\
        \phi\circ G_t\circ\phi^{-1}(x), & x \in \Uman.
    \end{cases}
\]
It follows that $\Sigma$ is fixed out of $\Uman$.
Moreover, due to~\ref{enum:Gprop:inv_circles}, $\Sigma$ must leave invariant the curves $\phi(\Circle\times\{s\})$ being contours of $\func$.
Hence $\Sigma_t\in\StabilizerIsotId{\func}\subset\xS$ for all $t\in[0;1]$.
Also, $\Sigma_0=\id_{\KleinBottle}$, while $\Sigma_1$ is fixed on $\alpha$ and thus belongs to $\StabilizerIsotId{\func,\alpha}=\xSa$.

Thus $\Sigma$ represents a path $\genSSa\colon\bigl([0;1],1,0\bigr)\to\bigl(\xS,\xSa,\id_{\KleinBottle}\bigr)$.
Recall that by~\eqref{equ:pi_k_S_Sa} we have the following isomorphisms:
\[
    \pi_1(\xS,\xSa) \xrightarrow{~\rho_{\alpha}~}
    \pi_1\DiffId(\alpha) \xrightarrow{~*~}  \pi_1 \alpha \cong \bZ
\]
where the first one is induced by the restriction map $\rho_{\alpha}\colon\Stabilizer{\func} \to \Diff(\alpha)$, $\rho_{\alpha}(\dif) = \restr{\dif}{\alpha}$, and the second one is induced by the inclu

Since $\restr{\Sigma_t}{\alpha}$ makes a full revolution of $\alpha$, it follows from Remark~\ref{rem:interpr_pi1_D_S1} that $\genSSa$ represents a generator of $\pi_1(\xS,\xSa)$.

Now, $\genSSa(1) = \Sigma_1$ is fixed on $\alpha$, and there exists another isotopy $\Delta\colon\KleinBottle\times[0;1]\to\KleinBottle$ such that $\Delta_0=\id_{\KleinBottle}$, $\Delta_t$ is fixed on $\alpha$, and $\Delta_1 = \Sigma_1$.
So it represents a path $\genInv\colon\bigl([0;1],1,0\bigr)\to\bigl(\xD,\xS,\id_{\KleinBottle}\bigr)$, and then $\eta(\genSSa) = \genInv\#\genSSa^{-1}$, so it is represented by the isotopy
\[
    K\colon\KleinBottle\times[0;1]\to\KleinBottle,
    \qquad
    K(x,t) =
    \begin{cases}
        \Delta(x,2t), & t\in[0;\tfrac{1}{2}], \\
        \Sigma(x,2-2t), & t\in[\tfrac{1}{2};1].
    \end{cases}
\]
being a loop in $\DiffId(\KleinBottle)$.
If $u\in\alpha$ is an arbitrary point, then the map $K^u\colon[0;1]\to\alpha \subset \KleinBottle$, $K^u(t)=K(u,t)$, represents a generator of $\pi_1\alpha$.
Hence by statement 3) of Corollary~\ref{cor:isot_klein_bottle}, $K$ is a generator of $\pi_1\DiffId(\KleinBottle)$.

Thus $\eta$ sends a generator of $\pi_1(\xS,\xSa)$ onto a generator of $\pi_1\xD$, and therefore $\eta$ is an isomorphism.
\end{proof}

For the proof of the second part of Theorem~\ref{intro:th:acylcic_mobius_orbit} we will also introduce several notations.
For $i=1,2$ let $\func_i = \restr{\func}{\kMob{i}}\colon\kMob{i}\to\bR$,
\begin{align*}
\xSM{i} &:= \Stabilizer{\func_i,\partial\kMob{i}}, &
\xDM{i} &:= \DiffFix{\func_i}{\partial\kMob{i}}, &
\xOM{i} &:= \Orbit{\func_i,\partial\kMob{i}}, \\
\xSanb  &:= \StabilizerNbh{\func,\alpha}, &
\xSMnb{1}  &:= \StabilizerNbh{\func_1,\partial\kMob{1}}, &
\xSMnb{2}  &:= \StabilizerNbh{\func_2,\partial\kMob{2}}.
\end{align*}

Denote by $\nbinc\colon\xSanb\subset\xSa$ and $\nbinc_i\colon\xSMnb{i}\subset\xSM{i}$, $i=1,2$, the natural inclusions, being by Lemma~\ref{lm:regular-invariant-neigh-homotopy-eq} homotopy equivalences.
Define also the following restriction maps
\begin{align*}
    \rstdif&\colon\xDa \to \xDM{1} \times \xDM{2}, &
    \rstdif(\dif) &= \bigl(\restr{\dif}{\kMob{1}},\restr{\dif}{\kMob{2}}\bigr),  \\
    \rstfunc&\colon \xOa \to \xOM{1} \times \xOM{2}, &
    \rstfunc(\gfunc) &= \bigl(\restr{\gfunc}{\kMob{1}},\restr{\gfunc}{\kMob{2}}\bigr).
\end{align*}
Finally, let $\actmap_i\colon\xDM{i}\to\xOM{i}$, $\actmap_i(\dif) = \func_i\circ\dif$, be the corresponding action map~\eqref{intro:equ:action} for $\kMob{i}$.
Then it is evident that the following diagram is commutative:
\[
    \xymatrix{
        \xDa \ar[r]^-{\rstdif} \ar[d]_-{\actmap} &
        \xDM{1} \times \xDM{2} \ar[d]^-{\actmap_1\times\actmap_2} \\
        \xOa \ar[r]^-{\rstfunc} & \xOM{1} \times \xOM{2}
    }
\]
\begin{lemma}\label{lm:rest_Oa_O1_O2}
The inclusion $\rstfunc\colon \xOa \to \xOM{1} \times \xOM{2}$ is a homotopy equivalence.
\end{lemma}
\begin{proof}
The proof consists of several reductions.

1) By Theorem~\ref{th:diff_act}\ref{enum:th:diff_act:OffX__Off}, $\OrbitComp{\func_i}{\func_i}$ coincides with $\OrbitComp{\func_i}{\func_i, \partial\kMob{i}} = \xOM{i}$, $i=1,2$.

2) Also, as $\xOa$, $\xOM{1}$, $\xOM{2}$ are aspherical, see~\eqref{equ:pik_Of}, it suffices to show that $\rstfunc$ induces an isomorphism of the corresponding fundamental groups $\rstfunc_1\colon\pi_1\xOa \to \pi_1\xOM{1} \times \pi_1\xOM{2}$

3) Note that $\actmap$ is a fibration with fiber $\xSa$, while $\actmap_1\times\actmap_2$ is a fibration with fiber $\xSM{1}\times\xSM{2}$.
Since $\xDa$, $\xDM{1}$, and $\xDM{2}$ are contractible and $\nbinc$, $\nbinc_1$ and $\nbinc_2$ are homotopy equivalences, we get the following commutative diagram:
\begin{equation}\label{equ:reduct_pi1O_pi0Snb}
\begin{gathered}
    \xymatrix{
        \pi_1\xOa \ar[r]^-{\rstfunc_1} \ar[d]_-{\partial}^-{\cong} &
        \pi_1\xOM{1} \times \pi_1\xOM{2} \ar[d]^-{\partial_1\times\partial_2}_-{\cong} \\
        \pi_0\xSa \ar[r]^-{\rstdif_0} & \pi_0\xSM{1} \times \pi_0\xSM{2} \\
        \pi_0\xSanb \ar[r]^-{\rstdif_0} \ar[u]^-{\nbinc}_-{\cong}& \pi_0\xSMnb{1} \times \pi_0\xSMnb{2} \ar[u]_-{\nbinc_1\times\nbinc_2}^{\cong}
    }
\end{gathered}
\end{equation}
in which the vertical arrows are isomorphisms.

But the restriction $\restr{\rstdif}{\xSanb}\colon\xSanb\to\xSMnb{1}\times\xSMnb{2}$, $\dif\mapsto\bigl(\restr{\dif}{\kMob{1}},\restr{\dif}{\kMob{2}}\bigr)$, is evidently an isomorphism of topological groups.
Hence the lower arrow $\rstdif_0$ in~\eqref{equ:reduct_pi1O_pi0Snb} is an isomorphism as well.
Therefore two other horizontal arrows~\eqref{equ:reduct_pi1O_pi0Snb} are also isomorphisms.
In particular, so is $\rstfunc_1$, which proves our lemma.
\end{proof}

\section*{Acknowledgments}
This work was supported by a grant from the Simons Foundation (SFI-PD-Ukraine-00014586, B.L.M and S.I.M.)


\def\cprime{$'$} \def\cprime{$'$} \def\cprime{$'$} \def\cprime{$'$}
  \def\cprime{$'$} \def\cprime{$'$} \def\cprime{$'$} \def\cprime{$'$}
  \def\cprime{$'$} \def\cprime{$'$} \def\cprime{$'$} \def\cprime{$'$}
  \def\cprime{$'$} \def\cprime{$'$}

\end{document}